\documentclass[12pt]{article}
\usepackage[stretch=10,shrink=10]{microtype}
\usepackage{amssymb,amsmath,amsthm,bm,lipsum}

\usepackage{authblk}
\usepackage{mathrsfs}

\usepackage{hyperref}
\hypersetup{
    colorlinks,
    citecolor=black,
    filecolor=black,
    linkcolor=black,
    urlcolor=black
}

\usepackage{color}

\newcommand{\blue}{\color{blue}}
\newcommand{\raph}{\blue}

\newcommand\blfootnote[1]{
	\begingroup
	\renewcommand\thefootnote{}\footnote{#1}%
	\addtocounter{footnote}{-1}%
	\endgroup
}

\begin{document}
\newcommand{\beq}{\begin{eqnarray}}
\newcommand{\eeq}{\end{eqnarray}}
\newcommand{\beas}{\begin{eqnarray*}}
\newcommand{\enas}{\end{eqnarray*}}
\newcommand{\bea}{\begin{eqnarray}}
\newcommand{\ena}{\end{eqnarray}}
\newcommand{\bms}{\begin{multline*}}
\newcommand{\ems}{\end{multline*}}
\newcommand{\qmq}[1]{\quad \mbox{#1} \quad}
\newcommand{\qm}[1]{\quad \mbox{#1}}
\newcommand{\nn}{\nonumber}
\newcommand{\bbox}{\hfill $\Box$}
\newcommand{\ignore}[1]{}
\newcommand{\bs}[1]{\boldsymbol{#1}}
\newcommand{\countingmeasures}{\math}
\newcommand{\I}[2]{I(#1;\,#2)}
\newcommand{\interior}[1]{\mathop{\mathrm{int}}#1}

\newcommand{\var}{\mathop{\mathrm{Var}}}
\newcommand{\RR}{\mathbb{R}}
\renewcommand{\S}{\mathbb{S}}
\renewcommand{\P}{\mathbb{P}}
\newcommand{\E}{\mathbb{E}}
\newcommand{\G}{\mathcal{G}}
\newcommand{\1}{\textbf{1}}
\newcommand{\norm}[1]{\lVert #1 \rVert}
\newcommand{\vol}[1]{| #1 |}
\newcommand{\Vol}{\text{Vol}}
\newcommand{\influential}{\textsc{Influential}}
\newcommand{\abs}[1]{\lvert #1 \rvert}
\newcommand{\Ppp}[1]{\mathscr{P}_{#1}}
\newcommand{\bpp}[1]{\mathscr{U}_{#1}}
\newcommand{\Bin}{\mathop{\mathrm{Bin}}}
\newcommand{\Poi}{\mathop{\mathrm{Poi}}}

\newcommand{\X}{\mathbb{X}}
\newcommand{\y}{\textbf{y}}
\newcommand{\x}{\textbf{x}}

\newcommand{\lcolor}[1]{\textcolor{magenta}
{#1}}  
\newcommand{\lcomm}[1]{\marginpar{\tiny\lcolor{#1}}}  
\newcommand{\toby}{\color[rgb]{0,0.5,0}}
\newcommand{\tcolor}[1]{%
\textcolor[rgb]{0,0.5,0}
{#1}}
\newcommand{\tcomm}[1]{\marginpar{\tiny\tcolor{#1}}}  
\newcommand{\extendedtcolor}{
\color[rgb]{0,0.5,0}
}

\newcommand{\rcomm}[1]{\marginpar{\tiny\raph{#1}}} 

\newcommand{\tr}{\mbox{tr}}
\newtheorem{theorem}{Theorem}
\newtheorem{corollary}[theorem]{Corollary}
\newtheorem{conjecture}[theorem]{Conjecture}
\newtheorem{proposition}[theorem]{Proposition}
\newtheorem{remark}[theorem]{Remark}
\newtheorem{lemma}[theorem]{Lemma}
\newtheorem{condition}[theorem]{Condition}
\newtheorem{assumption}{Assumption}
\renewcommand*{\theassumption}{\Alph{assumption}}
\newtheorem*{claim}{Claim}
\theoremstyle{definition}
\newtheorem{definition}[theorem]{Definition}
\newcommand{\pf}{\noindent {\bf Proof:} }
\newcommand{\equlaw}{\stackrel{(d)}{=}}

\title{{\bf\Large Bounds to the normal for proximity region graphs}}
\author[1]{Larry Goldstein}
\author[1]{Tobias Johnson}
\author[2]{Rapha\"el Lachi\`eze-Rey}
\affil[1]{University of Southern California} \affil[2]{Universit\'e Paris Descartes, Sorbonne Paris Cit\'e}

\maketitle

\begin{abstract}
In a proximity region graph ${\cal G}$ in $\mathbb{R}^d$, two distinct points $x,y$ of a point process $\mu$ are connected when the `forbidden region' $S(x,y)$ these points determine has empty intersection with $\mu$. The Gabriel graph, where $S(x,y)$ is the open disc with diameter the line segment connecting $x$ and $y$, is one canonical example. When $\mu $ is a Poisson or binomial process, under broad conditions on the  regions $S(x,y)$, bounds on the Kolmogorov and Wasserstein distances to the normal are produced for functionals of ${\cal G}$, including the total number of edges and the total length. Variance lower bounds, not requiring strong stabilization, are also proven to hold for a class of such functionals.
\end{abstract}

\footnotetext[1]{The work of the first author was partially supported by NSA-H98230-15-1-0250, and of the second author by NSF grant DMS-1401479}

\blfootnote{MSC 2010 subject classifications:
	60D05\ignore{Geometric probability and stochastic geometry}, 52A22\ignore{Random sets and integral geometry}, 60F05\ignore{Central limit and other weak theorems}} 

\blfootnote{Key words and
	phrases: Forbidden region graph, Berry-Esseen bounds, stabilization, Poisson functionals}

\section{Introduction}
The family of graphs that we study here, all with vertex sets given by a locally finite point process $\mu$ in $\mathbb{R}^d$, is motivated by two canonical examples considered in \cite{AlSh10}, the Gabriel graph and the relative neighborhood graph. Two distinct points $x$ and $y$ of $\mu$ are connected by an edge in the Gabriel graph if and only if there does not exist any point $z$ of the process $\mu$ lying in the open disk whose diameter is the line segment connecting $x$ and $y$. The relative neighborhood graph has an edge between $x$ and $y$ if and only if there does not exist a point $z$ of $\mu$ such that 
$$
\max(\|x-z\|, \|z-y\|) < \|x-y\|,
$$
that is, if and only if there is no point $z$ of $\mu$ that is closer to either $x$ or $y$ than these points are to each other.

These two examples are special cases of `proximity graphs' as defined in \cite{Dev88}, where distinct points $x$ and $y$ of $\mu$ are connected if and only if a region $S(x,y)$ determined by $x$ and $y$ contains no points of $\mu$, that is, when $\mu \cap S(x,y) = \emptyset$. As $S(x,y)$ must be free of points of $\mu$ in order for $x$ and $y$ to be joined, we call $S(x,y)$ the `forbidden region' determined by $x$ and $y$. 
In particular, with $B(x,r)$ and $B^o(x,r)$ denoting the closed and open ball of radius $r$ centered at $x$, respectively, the forbidden regions of the Gabriel graph are given by
\bea \label{GabrielSxy}
S(x,y)=B^o((x+y)/2,\|y-x\|/2),
\ena
and those of the relative neighborhood graph by
\bea \label{rngSxy}
S(x,y)=B^o(x,\|y-x\|) \cap B^o(y,\|x-y\|).
\ena
It is easy to check that the forbidden regions $S(x,y)$ of the Gabriel graph are contained in those of the relative neighbor graph, and hence edges of latter are also edges of former.

We refer to the graphs formed in this manner also as `forbidden region graphs'. Indeed, when coining the label `proximity graphs' in \cite{Dev88}, one reads that `this term could be misleading in some cases.' Indeed, forbidden region graphs may depend on `non-proximate' information, such as the graph considered in Example 5 of \cite{Dev88}, whose forbidden region $S(x,y)$ is the infinite strip bounded by the two parallel hyperplanes containing $x$ and $y$, each perpendicular to $y-x$. Allowing forbidden regions to depend on larger sets of points and to be determined by more complex rules yield well studied graphs with additional structure, including the Minimum Spanning Tree and the Delaunay triangulation, see  \cite{AlSh10}.

For a forbidden region graph $\G$ and a Poisson or binomial point process $\mu$ in some bounded measurable `viewing window' denoted $\X$ in the sequel, ensuring that the graph and functional $L(\mu)$ in \eqref{eq:defLalpha}  is finite, we study the distribution of 
\bea \label{eq:defLalpha} 
L(\mu) = \sum_{\{x,y\} \subseteq \mu , x\not = y}{\bf 1}(\mu \cap S(x,y) = \emptyset)\psi(x,y),
\ena
for some $\psi\colon\mathbb{R}^d \times \mathbb{R}^d \rightarrow \mathbb{R}$ satisfying $\psi(x,y)=\psi(y,x)$. For instance, taking $\psi(x,y)=\|x-y\|^\alpha$ for some $\alpha \ge 0$, for $\alpha=0$ and $\alpha=1$ the value of $L(\mu)$ is the number of edges and the total length of ${\cal G}$, respectively. 

Recall that the Kolmogorov distance between random variables $U$ and $V$ is defined as
\begin{align*}
d_{K}(U,V)=\sup_{t\in \mathbb{R}}\left|
\P(U\leqslant t)-\P(V\leqslant t)
\right|,
\end{align*}and the Wasserstein distance as
\begin{align*}
d_{W}(U,V)=\sup_{h\in \text{\rm{Lip}}_{1}}\left|
\E[h(U)-h(V)]
\right|,
\end{align*}
where $\text{\rm{Lip}}_{1}$ stands for the class of $1$-Lipschitz functions $\mathbb{R}\to \mathbb{R}$.
Theorem \ref{thm:main}, our main result, is a bound on the normal approximation of $L$ in $d(\cdot,\cdot)$, denoting either the Wasserstein or Kolmogorov metric,
that holds under broad conditions on the forbidden regions and underlying point process. Its immediate corollary, in conjunction with the variance lower bound of Theorem \ref{prop:variance}, provides the following result for the two motivating examples just introduced. 

\begin{corollary}
	Let $\X=B(0,1)$, and suppose that $\eta_t$ is either a Poisson process  with intensity~$t$
	on $\X$, or a binomial process of $t$ independent and uniformly distributed points on $\X$, and let $F_{t}:=L(\eta_t)$ for $t \ge 1$, where $L(\cdot)$ is given in \eqref{eq:defLalpha} with
	$\psi(x,y)=\|x-y\|^\alpha$ for some $\alpha \ge 0$. Then for the Gabriel graph and the relative neighborhood graph, there exists a constant $C >0$ such that
	\begin{align*} 
	d(\widetilde F_{t},N) \leq Ct^{-1/2} \qm{for all $t \ge 1$,}
	\end{align*}
	where $\widetilde F_t = (F_t-\E F_t)/\sqrt{\var F_t}$, and $N$ is a standard Gaussian.
\end{corollary}

Proximity graphs arise frequently in stochastic geometry, especially for their use in communication networks, see \cite{AlSh10,Dev88,PenYuk01,Schreiber} and references therein. For most models, first and second order limit theorems were already known from the theory of stabilizing functionals (see \cite{PenYuk01,Schreiber}), but obtaining optimal speed of convergence and confidence intervals remained open. With Poisson input,   the general results available  only give a non-optimal rate of convergence, while the speed obtained here
is typically optimal for stationary stabilizing functions. For integer-valued functionals
like the number of edges, the methods of \cite{englund1981remainder} imply the rate here is optimal whenever the rates of upper and lower variance bounds agree. 
Only recently was an optimal rate of convergence established for \emph{any} geometric functional
with non-deterministic range of interaction, when \cite{last2014normal} did so for statistics of the
nearest neighbor graph and of the Poisson--Voronoi tessellation.
 Furthermore, those results are only valid with Poisson input, whereas ours also hold for binomial input.
In general, with binomial input there are few preexisting results on geometric functionals as considered here. Though one can most likely derive asymptotic normality from \cite{PenYuk01}, no speed of convergence was available at all for the models considered in this paper.  See \cite{GoldsteinPenrose} for an optimal speed for the Boolean model, where the interaction range is bounded. Chatterjee also gives a slower power law decay  for some nearest-neighbor statistics in \cite{Cha08}. Our paper is the first example of an optimal speed of convergence for geometric functionals with possibly complex dependency structure between points, and no prior bound on the speed of convergence, when the input consists of $n$ i.i.d.\ points uniformly distributed in a square of volume $n$.

In our consideration of more general graphs, we will assume that the collection of forbidden regions $\{S(x,y): \{x,y\} \subseteq \mathbb{R}^d, x \not =y \}$ 
consists of nonempty measurable subsets of $\mathbb{R}^d$ that are symmetric in that
\bea \label{S.symmetric}
S(x,y)=S(y,x) \qmq{for all $\{x,y\} \subseteq  \mathbb{R}^d, x \not =y$.}
\ena
Nonsymmetric sets $S(x,y)$ would be natural for the construction of directed forbidden region graphs, and though we do not consider them here our methods would apply.
With $\overline{S}$ denoting the closure of a set $S \subseteq \mathbb{R}^d$, we assume also that
\bea \label{eq:S.basic.prop}
\{x,y\} \subseteq \overline{S(x,y)} \setminus S(x,y),
\ena 
and that the \emph{normalized diameter} ${\cal D}$ of the collection of forbidden regions is finite, that is,

\bea \label{eq:def.diameter}
{\cal D}<\infty \qmq{where} {\cal D}=\sup\left\{\frac{\|s-t\|}{\|x-y\|}: \{s,t\} \subseteq S(x,y),\{x,y\}\subseteq \mathbb{R}^{d}, x \not = y \right\}.
\ena

Assumption~\ref{assumption:scaledball} below
requires that as $x$ and $y$ become farther apart, the forbidden regions $S(x,y)$ contain increasingly large balls. 
Note, for instance, that if all forbidden regions have empty interior, then the graph determined by a Poisson or binomial input process with an absolutely continuous intensity measure would be the complete graph almost surely. In Assumptions \ref{assumption:scaledball} and \ref{assumption:PoissonDom} and Theorem \ref{thm:main}, $\X$ will denote a window specified by a given bounded measurable subset of $\mathbb{R}^d$.
\begin{assumption}[Scaled ball condition]\label{assumption:scaledball}
	For some $\delta>0$ and window $\X$, it holds for all $\{x,y\} \subseteq \X$ that $S(x,y)\cap\X$ contains a ball of radius $\delta\norm{x-y}$.
\end{assumption}

With some slight abuse of notation, $| \cdot |$ will be used to denote both the Lebesgue measure of a measurable subset of $\mathbb{R}^d$ and also cardinality of a finite set; use will be clear from context. Our results below provide bounds on the normal approximation of $L$ in \eqref{eq:defLalpha} when the underlying graph is generated by a point process $\eta_t, t>0$ that satisfies the following conditions.

\begin{assumption} \label{assumption:PoissonDom}
	Let $\lambda$ be a probability measure on $\X$ satisfying
	\beas
	c_\lambda \vol{B} \le  \lambda(B) \le b_\lambda \vol{B} \qm{for all measurable $B \subseteq \X$}
	\enas
	for some $0<c_\lambda \le b_\lambda$.
	The point process $\eta_t$ is either a Poisson process $\Ppp{t}$ on $\X$ with intensity~$\lambda_t=t\lambda, t>0$,
	or a binomial process $\bpp{t}$ consisting of a set of i.i.d.\ variables $X_{1},\dots ,X_{t}$ with common distribution $\lambda$, for
	$t\in\mathbb{N}$.
\end{assumption}

Lastly, we require the following variance lower bound.
\begin{assumption} \label{assumption:sigh}
For $\alpha \geqslant 0$, there exists $v_\alpha > 0$ such that
\begin{align*}
 \var L(\eta_t) \geq v_\alpha t^{1-2\alpha/d} \qmq{for all $t \ge 1$.} 
\end{align*}
\end{assumption}
Assumption \ref{assumption:sigh} is a serious one, and we separately address the question of when it is satisfied in Section \ref{sec:vlb}, see Theorem \ref{prop:variance} below.

We inform the reader that the $C$ that appears in our bounds denotes a positive constant that may not be the same at each occurrence.

\begin{theorem}
 \label{thm:main}
For a given window $\X$, let 
$\{S(x,y)\colon x,y\in\X,\,x\neq y\}$
be a collection of forbidden regions satisfying \eqref{S.symmetric}--\eqref{eq:def.diameter}, and let
Assumption~\ref{assumption:scaledball} hold.
Let $\eta_{t}$ be a point process on $\X$ satisfying Assumption~\ref{assumption:PoissonDom}, and let
\begin{align*}
F_{t}:=L(\eta_t),\qm {for $t \ge 1$,}
\end{align*} 
where $L(\cdot)$ is given in \eqref{eq:defLalpha}, where for some $C >0$ and $\alpha \ge 0$ we have $|\psi(x,y)| \le C \|x-y\|^\alpha$ for all $\{x,y\} \subseteq \mathbb{R}^d$.

If Assumption \ref{assumption:sigh} holds, then with $d(\cdot,\cdot)$ denoting either the Wasserstein or Kolmogorov distance, there exists a constant $C$ not depending on $t$ such that
\begin{align*} 
d(\widetilde F_{t},N) \leq Ct^{-1/2} \qm{for all $t \ge 1$,}
\end{align*}
where $\widetilde F_t = (F_t-\E F_t)/\sqrt{\var F_t}$, and $N$ is a standard Gaussian. 
\end{theorem}

Theorem~\ref{thm:main} is based on the methods of \cite{last2014normal}, in particular on second order Poincar\'e inequalities, and also the key notion of stabilization.
To define stabilization, let $f(\mu)$ be a function of a point process $\mu$ in $\mathbb{R}^d$. 
For $x \in \mathbb{R}^d$ consider the difference (or derivative) at $x$ given by 
\bea \label{def:Dsubx}
D_x f(\mu) = f(\mu \cup \{x\}) - f(\mu), 
\ena
which is the amount that $f$ changes upon the insertion of the point $x$ into $\mu$. Higher order differences are defined iteratively, for instance $D_{x,y}^2f(\mu) = D_x( D_y f(\mu))$, so
\bea \label{Dxy.explicit}
D_{x,y}^2f(\mu)= f(\mu \cup \{x,y\})-f(\mu \cup \{y\})-f(\mu \cup \{x\} ) + f(\mu). 
\ena

There are a  number of related notions of a stabilization radius for a functional~$f$.
The one we will use is a radius $R(x;\mu)$ such that
\begin{align}
D^2_{x,y}f(\mu) =0 \qquad\text{if $\norm{y-x}>R(x;\mu)$}.\label{eq:stabR}
\end{align}
We say in this case that
$R(x;\mu)$ is \emph{{  a stabilization radius}} for $f$ around $x$.

When dealing with a function of a Poisson process $\Ppp{t}$ with growing intensity $\lambda_t$, one key condition from \cite{last2014normal} required to obtain bounds to the normal for a properly standardized functional $f$ is that over the observation window $\X$, 
\bea \label{int.P.to.a.finite}
\sup_{x \in \X, t \ge 1} \int \P(D_{x,y}^2 f(\Ppp{t}) \not = 0)^a \lambda_t(dy)< \infty,
\ena
for $a$ some small number, depending on low moments of the derivatives of $f$.
If there exists a stabilization radius for $f$ that is small with sufficiently high probability,
then \eqref{int.P.to.a.finite} holds. In Section~\ref{sec:stabilization}, we construct such a radius and prove that it exhibits exponential decay under very weak conditions on the forbidden regions.

We now address Assumption~\ref{assumption:sigh} in Theorem \ref{thm:main}, the lower bound on $\var L(\eta_t)$.
Penrose and Yukich give a general lower bound for the variance of Poisson and binomial statistics 
in \cite{PenYuk01}. Their result requires a statistic to be \emph{strongly stabilized}.
(This notion of stabilization is also referred to as stabilization for add-one cost
or as external stabilization---see \cite{Schreiber} for a general survey.)
We cannot apply this result because our statistic $L$ is not strongly stabilized 
unless we impose additional constraints on the forbidden regions, such as requiring them to be convex.
Another possible approach would be to use the results of \cite[Section~5]{last2014normal}.
These are applicable to $L$, but only for the easier case of Poisson input. 
We are thus forced to give a new argument to prove that Assumption~\ref{assumption:sigh} holds in some generality. We state Assumption \ref{assumption:specialy}, an additional technical condition required, followed by Theorem \ref{prop:variance}, providing sufficient conditions for Assumption~\ref{assumption:sigh}. For a simple statement we restrict ourselves to the regular isotropic case as specified by Definition \ref{df:isotropic}, but a more general result can be formulated on conditions that make the expectation  \eqref{eq:more.general.than.iso} zero.
Let $\partial B$ denote the boundary of a set $B \subseteq \mathbb{R}^d$.

\begin{assumption} \label{assumption:specialy}
For all $\{w,z\}\subseteq B(0,1)$ there exists
$y\in\partial S(w,z)$ such that
$z\notin\partial S(w,y)$ and $w\notin\partial S(z,y)$.  We furthermore assume that the choice $(w,z)\mapsto y\in \partial S(w,z)$ can be made in a measurable way.
\end{assumption}

To state the variance lower bound, we work in a setup where the forbidden region $S(x,y)$
is given by a template shifted and scaled according to $x$ and $y$.
\begin{definition}[Regular isotropic family] \label{df:isotropic}
Let $S\subseteq\RR^d$ be a bounded, measurable set symmetric around an axis 
given by a unit vector
$u_0\in\RR^d$; that is, any rotation leaving $u_0$ invariant also leaves $S$ invariant.
Assume that rotations taking $u_0$ to $-u_0$ leave $S$ fixed and that
$\{u_0,-u_0\}\subseteq 2(\overline{S}\setminus S)$. Also assume that
$S$ contains an open ball and has negligible boundary.

Given $x,y\in\RR^d$ with $x\neq y$, let $\rho_{xy}$ be the rotation tranforming $u_0$ into 
$(x-y)/\norm{x-y}$ and leaving invariant the orthogonal complement of the space spanned
by $\{u_0,x-y\}$. Then, define
\begin{align*}
  S(x,y) &= (x+y)/2 + \norm{x-y}\rho_{xy}(S).
\end{align*}
We call the resulting
collection of forbidden regions a \emph{regular $(S,u_0)$ isotropic family}.
\end{definition} 
Because $S$ is symmetric around $u_0$, we could have taken $\rho_{xy}$ to be any rotation
transforming $u_0$ into $(x-y)/\norm{x-y}$ without affecting the final definition of $S(x,y)$.
Indeed, if $\rho$ and $\rho'$ are any two rotations tranforming $u_0$ into $(x-y)/\norm{x-y}$,
then the rotation $\rho^{-1}\rho'$ leaves $u_0$ invariant and hence also leaves $S$ invariant.
Thus $\rho' (S)=\rho\rho ^{-1}(\rho' (S))=\rho(S)$.
We also mention that in $\RR^2$, the vector $u_0$ is irrelevant and $S$ need not have
any rotational invariance, since the only rotation leaving $u_0$ invariant is the identity, which
necessarily leaves $S$ invariant as well.

One should think of $S(x,y)$ in a regular $(S,u_0)$ isotropic family as being generated 
by translating $S$ to the midpoint of $x$ and $y$, then rotating $S$ according to the orientation
of $x$ and $y$, and then scaling $S$ according to the distance between $x$ and $y$.
Our assumptions that rotations taking $u_0$ to $-u_0$ leave $S$ fixed and that 
$\{u_0,-u_0\}\subseteq 2(\overline{S}\setminus S)$ ensure that the family satisfies
properties~\eqref{S.symmetric} and \eqref{eq:S.basic.prop}.
Later in this introduction, we will show that the forbidden regions of 
our two canonical examples, the Gabriel graph and
the relative neighborhood graph, are regular isotropic families.

\begin{theorem}\label{prop:variance}
  Suppose the forbidden regions $\{S(x,y): \{x,y\} \subseteq \mathbb{R}^d, x \not = y\}$ form a regular $(S,u_0)$ isotropic family and satisfy Assumption~\ref{assumption:specialy}. Assume further that
    the scaled ball condition, Assumption~\ref{assumption:scaledball}, is satisfied with the role of $\X$ played by $t^{1/d}\X\cap B(x,r)$ for a fixed $\delta>0$ for any $t$ and $r$, that $\X$ is star shaped with star center at the origin, and that it contains an open set around the origin. For the function $\psi$ in  the definition \eqref{eq:defLalpha} of $L$, assume
   \begin{itemize}
    \item $\psi(ax,ay)=a^\alpha \psi(x,y)$ for all $a>0$ and $\{x,y\} \subseteq \mathbb{R}^d$ 
    \item $\psi(x,y)\not = 0$ for all $x \not =y$
    \item $\psi(x,y)$ is continuous on $\mathbb{R}^{d}\times \mathbb{R}^{d}$.
    \end{itemize}
   Then there is a constant $v_\alpha>0$ such that Assumption \ref{assumption:sigh}
    holds when $\eta_t$ is either a homogeneous Poisson process on $\X$ with intensity $t$ or
    a binomial point process of $t$ independent and uniformly distributed points on $\X$.
\end{theorem}

We end this section by introducing some additional terminology about forbidden regions
and regular isotropic families. As already stated, the graph with vertex set a locally finite point configuration $\mu$ in $\mathbb{R}^d$ is the $S(x,y)$ forbidden region graph on $\mu$ when an edge exists between points $x$ and $y$ of $\mu$ if and only if $x \not = y$ and $S(x,y) \cap \mu = \emptyset$. That is, we connect points $x$ and $y$ of $\mu$ if and only if they are distinct, and there are no points of $\mu$ lying in the forbidden region $S(x,y)$ that these two points generate.
Hence, for $x \in \mu$, the set of edges $\mathcal{G}_{S}(x;\mu)$ incident to $x$ in $\mu$, and the edge set $\mathcal{G}_{S}(\mu)$ of the forbidden region graph are given, respectively, by 
\bea \label{def:GSedge.set}
\G_{S}(x;\mu)=\{\{x,y\}:\{x,y\} \subseteq \mu, x \not = y, S(x,y)\cap \mu =\emptyset \} \qmq{and} \G_{S}(\mu ) = \bigcup_{x \in \mu} \G_{S}(x;\mu).
\ena
We may drop the subscript when the dependence on $S$ is clear from context.


We call a collection $S(x,y)$ of forbidden regions \emph{translation invariant} when
\beas 
S(x+z,y+z)=S(x,y)+z \qmq{for all $\{x,y,z\} \subseteq  \mathbb{R}^d, x \not = y$,}
\enas
and we observe that a regular isotropic family is always translation invariant.
The normalized diameter \eqref{eq:def.diameter} for a regular $(S,u_0)$ isotropic family
is given by
\beas
{\cal D}= 
\sup\{\|y-x\| : \{x,y\} \subseteq   S\}.
\enas

Our two canonical examples, the Gabriel graph and the relative neighborhood graph, are both 
regular isotropic families. With $u_{0}=(1,0,\dots,0)$, the  Gabriel graph is obtained by setting $S=B^o(0,1/2)$, and the relative neighborhood graph by $S=B^{o}(u_0/2,1)\cap B^{o}(-u_0/2,1)$.
For the Gabriel graph, we then have
\beas
S(x,y)= (x+y)/2 + \norm{x-y}\rho_{xy}\bigl(B^o(0,1/2)\bigr) = (x+y)/2 + \norm{x-y}B^o(0,1/2),
\enas
which agrees with \eqref{GabrielSxy}.
Rotating the template $S$ given above for the relative neighborhood graph, we have
\begin{align*}
  \rho_{xy}(S) &= B^o\bigl((x-y)/2\norm{x-y}, 1\bigr) \cap B^o\bigl((y-x)/2\norm{y-x}, 1\bigr),
\end{align*}
which yields
\begin{align*}
  S(x,y) &= (x+y)/2 +\norm{x-y}\rho_{x,y}(S)\\
         &= (x+y)/2 + B^o((x-y)/2,\norm{x-y}) \cap B^o((y-x)/2,\norm{y-x})\\
         &= B^o(x,\|y-x\|) \cap B^o(y,\|x-y\|),
\end{align*}
agreeing with \eqref{rngSxy}.

\section{Radius of stabilization}\label{sec:stabilization}
We begin this section by constructing a set in \eqref{def:calRsubS} that will serve as a stabilizing region about a point $x \in \mathbb{R}^d$, or more generally around a subset $U \subseteq \mathbb{R}^d$. Our radius $R_S(U;\mu;\X)$ is then constructed in \eqref{def:Rs} in terms of this set. We prove in Lemma \ref{lem:monotonicity} that $R_S(U;\mu;\X)$ is monotone in $\mu$, and in Lemma~\ref{lem:stabilization} that it is a stabilization radius for $L$ around $x$ as defined in \eqref{eq:stabR}.
In Proposition \ref{prop:Rs.exponential.bound}, we show that the stabilization radius has exponentially
decaying tails with standard Poisson or binomial input under
Assumption \ref{assumption:scaledball}, the scaled ball condition, on the forbidden regions. We remind the reader that $\X \subseteq \mathbb{R}^d$ is a bounded measurable window.

For $U \subseteq \mathbb{R}^d$, let 
\begin{multline} \label{def:calRsubS}
  \mathcal{R}_S(U; \mu; \X) \\= \bigcup\Big\{ S(w,z): \text{$\{w,z\}\subseteq\X$ such that
    $S(w,z)\cap\mu=\emptyset$ and $U\cap\overline{S(w,z)}\neq\emptyset$} \Big\}.
\end{multline}
Intuitively, this set consists of all forbidden regions affected by the addition of a point somewhere in
$U$. The most important case for us is $U=\{x\}$, which we write as
$\mathcal{R}_S(x; \mu;\X)$.

First, we show $\mathcal{R}_S(U; \mu; \X)$ satisfies a monotonicity property in $\mu$.
\begin{lemma}\label{lem:monotonicity}
  If $\mu\subseteq\nu$, then
  \begin{align*}
    \mathcal{R}_S(U; \nu; \X)\subseteq\mathcal{R}_S(U; \mu; \X),
  \end{align*}
  with equality if $\nu\setminus\mu$ lies outside of $\mathcal{R}_S(U; \mu; \X)$.
\end{lemma}
\begin{proof}
  Suppose that $S(w,z)$ satisfies $S(w,z)\cap\nu=\emptyset$ and $U\cap \overline{S(w,z)}\neq\emptyset$.
  Then this forbidden region also satisfies $S(w,z)\cap\mu=\emptyset$, showing that
  $\mathcal{R}_S(U;\nu;\X)\subseteq\mathcal{R}_S(U;\mu;\X)$.
  
  Now, assume that $\nu\setminus\mu$ lies outside of $\mathcal{R}_S(U; \mu; \X)$.
  Suppose that $S(w,z)$ satisfies $S(w,z)\cap\mu=\emptyset$ and $U\cap\overline{S(w,z)}\neq\emptyset$.
  Then $S(w,z)\subseteq \mathcal{R}_S(U;\mu;\X)$, and hence $\mu=\nu$ on $S(w,z)$.
  This implies that $S(w,z)\cap \nu=\emptyset$, which means that $S(w,z)\subseteq\mathcal{R}_S(U;\nu;\X)$.
  Therefore $\mathcal{R}_S(U;\mu;\X)\subseteq\mathcal{R}_S(U;\nu;\X)$, proving the two sets equal.
\end{proof}

  Now we consider the relation between $\mathcal{R}_S(U;\mu;\X)$ and the graphs $\G(\mu)$ and $\G(\mu \cup \{x\})$.
  Let $E^+_x(\mu)$ denote the edges found in $\G(\mu\cup\{x\})$ but not in $\G(\mu)$, and let
  $E^-_x(\mu)$ denote the edges found in $\G(\mu)$ but not in $\G(\mu\cup\{x\})$, that is
 \bea \label{def:Epm}
 E_x^+ (\mu)=\G(\mu \cup \{x\}) \setminus \G(\mu) \qmq{and} E_x^-(\mu)=\G(\mu) \setminus \G(\mu \cup \{x\}).
 \ena
  \begin{lemma}\label{lem:same.change.graph}
    Suppose that $\mu$ and $\nu$ are
    supported on some bounded measurable windows $\X_1$ and $\X_2$, respectively, and that $U\subseteq \X_1\cap \X_2$.
    If $\mathcal{R}_S(U;\mu;\X_1)=\mathcal{R}_S(U;\nu;\X_2)$ and $\mu$ and $\nu$ agree on the closure of this set, then
    $E^{\pm}_x(\mu)=E^{\pm}_x(\nu)$ for any $x\in U$.
  \end{lemma}
  \begin{proof}
    Suppose that $x\in U$ and $\{x,y\}\in E^+_x(\mu)$. 
    Then, by $\{x,y\}\subseteq \X_1$, $S(x,y)\cap\mu=\emptyset$ and \eqref{eq:S.basic.prop}, we have $S(x,y)\subseteq\mathcal{R}_S(U;\mu;\X_1)=\mathcal{R}_S(U;\nu;\X_2)$.
  Again by \eqref{eq:S.basic.prop} the closure of this set contains $y$, and $\mu$ and $\nu$
    agree on it. Thus $y\in\nu$ and $S(x,y)\cap\nu=\emptyset$, implying that
    $\{x,y\}\in E^+_x(\nu)$.
    Therefore $E^+_x(\mu)\subseteq E^+_x(\nu)$. By symmetry, the opposite inclusion holds as well.
    
   Now suppose that $x \in U$ and $\{w,z\} \in E_x^-(\mu)$. As $\{w,z\} \in \G(\mu)$ we have $\{w,z\} \subseteq \X_1$ and $S(w,z) \cap \mu =\emptyset$, and as $\{w,z\} \not \in \G(\mu \cup \{x\})$ we must have $x \in S(w,z)$, so that $U \cap \overline{S(w,z)} \supseteq \{x\} \not = \emptyset$.  Hence $S(w,z) \subseteq \mathcal{R}_S(U;\mu;\X_1)$, and so is also a subset of $\mathcal{R}_S(U;\nu;\X_2)$. As $\nu$ agrees with $\mu$ on the closure of this set we have $\{w,z\} \subseteq \X_2$ and $S(w,z) \cap \nu = \emptyset$, so $\{w,z\} \in \G(\nu)$. As $\{w,z\} \not \in \G(\mu \cup \{x\})$ we have $x \in S(w,z)$, and therefore $S(w,z) \cap (\nu \cup \{x\})=\{x\}$. Hence $\{w,z\} \not \in \G(\nu \cup \{x\})$, showing $E_x^-(\mu) \subseteq E_x^-(\nu)$. By symmetry, the opposite inclusion also holds. 
  \end{proof}

Next, for $U\subseteq\X$ and $\mu$ supported on $\X$, define
\begin{align} 
R_S(U;\mu;\X) =\sup\big\{\norm{y-x}: y\in\mathcal{R}_S(U;\mu;\X),\,x\in U\big\},\label{def:Rs}
\end{align}
writing this quantity as $R_S(x;\mu;\X)$ if $U=\{x\}$. The next lemma shows that $R_S(x;\mu;\X)$ is {  a stabilization radius.}
\begin{lemma} \label{lem:stabilization}
The radius $R_S(U;\mu;\X)$ given in \eqref{def:Rs} is stabilizing in the sense of \eqref{eq:stabR} for $L(\mu)$, the statistic defined in
\eqref{eq:defLalpha}. 
That is,  
\begin{align*}
D_{x,y}^{2}L(\mu)=0 \quad \mbox{for all $\{x,y\} \subseteq \mathbb{X}$ with $x\in U$ and
$\|y-x\| > R_S(U;\mu;\X)$.}
\end{align*}
Furthermore, for $x\in U$ and $\{x_1,\ldots,x_n\}\subseteq\X$ satisfying $\norm{x_i-x}>R_S(U;\mu;\X)$,
\begin{align}\label{eq:add.points}
D_xL(\mu)=D_xL(\mu\cup\{x_1,\ldots x_n\}).
\end{align}

\end{lemma}
\begin{proof}
  Assume that $x\in U$ and $\|y-x\| > R_S(U;\mu;\X)$. We need to show
  that $D_x L(\mu\cup\{y\}) = D_xL(\mu)$. To do so, we will show that
  $E^{\pm}_x(\mu\cup\{y\})=E^{\pm}_x(\mu)$.
  Since $\|y-x\| > R_S(U;\mu;\X)$, the point $y$ lies outside of $\overline{\mathcal{R}_S(U;\mu;\X)}$.
  By Lemma~\ref{lem:monotonicity}, $\mathcal{R}_S(U;\mu\cup\{y\};\X) = \mathcal{R}_S(U;\mu;\X)$.
  On the closure of this set, $\mu$ and $\mu\cup\{y\}$ agree, and so applying
  Lemma~\ref{lem:same.change.graph} with $\nu=\mu\cup\{y\}$ and $\X_1=\X_2=\X$ yields the first conclusion.
  
Now, we will repeatedly apply this first conclusion to establish \eqref{eq:add.points}.
  Applying it once shows that since $\norm{x_1-x}>R_S(U;\mu;\X)$,
\begin{align*}
D_xL(\mu)= D_xL(\mu\cup\{x_1\}).
\end{align*}
By Lemma~\ref{lem:monotonicity}, we have $R_S(U;\mu\cup\{x_1\};\X)\leq 
R_S(U;\mu;\X)$. Thus applying the first claim again yields
\begin{align*}
D_xL(\mu\cup\{x_1\})=D_xL(\mu\cup\{x_1,x_2\}).
\end{align*}
Repeating this argument proves \eqref{eq:add.points}.
\end{proof}

To prove that our stabilization radius has exponential tails under Poisson or binomial input,
the rough idea is that if the stabilization radius is large, then there must be a large ball
empty of points of $\mu$.
\begin{lemma}\label{lem:emptyspace}
  Assume that  $\X$ and the collection of forbidden regions $S(x,y)$ satisfy the scaled ball condition (Assumption~\ref{assumption:scaledball}) with $\delta>0$, and let $\mu$ be supported on $\X$. If for some ${  u\in\X,}  r \ge 0$ and $0<r_1<r_2$ we have $B(u,r)\subseteq \X$ and $0<r_1< R_S(B(u,r);\mu;\X)\leq r_2$, then with ${\cal D}$ the normalized diameter in \eqref{eq:def.diameter}, there exists a ball of radius $(r_1-2r)\delta /{\cal D}$ lying within
  $B(u,r_2)\cap\X$ that contains no points of $\mu$. 
\end{lemma}
\begin{proof}
  Since $R_S(B(u,r);\mu;\X)> r_1$, there exist $\{w,z\}\subseteq \X$
  such that
  
  \begin{itemize}
    \item $S(w,z)$ contains no points of $\mu$;
    \item $\overline{S(w,z)}$ contains some point of $B(u,r)$;
    \item and there exists $y\in S(w,z)$ and $x\in B(u,r)$
      with $\norm{y-x}>r_1$.
  \end{itemize}
  \noindent The diameter of $S(w,z)$ is then greater than $r_1-2r$ by the 
  triangle inequality, and by the definition
  of the normalized diameter~${\cal D}$, we have $\norm{z-w}>(r_1-2r)/{\cal D}$. 
  By the scaled ball condition, $S(w,z)\cap\X$ contains a ball of radius $\delta (r_1-2r)/{\cal D}$.
  Since $R_S(B(u,r);\mu;\X) \leq r_2$, the set $S(w,z)$ is contained within $B(u,r_2)$ (in fact,
  it is contained in $B(u,r_2-r)$, but we will not need this fact), and so
  the ball is also contained within $B(u,r_2)\cap \X$. By virtue of being a subset of
  $S(w,z)$, the ball contains no points of $\mu$.
\end{proof}

Using Lemma~\ref{lem:emptyspace} we now show our stabilization radius has exponential tails. 
\begin{proposition}\label{prop:Rs.exponential.bound}
If the scaled ball condition (Assumption~\ref{assumption:scaledball}) holds for $\delta>0$, and 
$\eta_t$ satisfies Assumption~\ref{assumption:PoissonDom} with $c_\lambda>0$, then for any { $x\in\X,$}
$0\leq\epsilon < 1/2$ and $r$ such that $B(x,\epsilon r)  \subseteq \X$,
  \begin{align}
    \P(R_{S}(B(x,\epsilon r);\eta_t;\X) \geq r)\leq C (1-2\epsilon)^{-d}\exp(-c_\lambda \kappa tr^{d}) \qm{for all $r > 0$}\label{eq:Rs.exponential.bound.ball}
  \end{align}
  with $\kappa=((1-2\epsilon)\delta/{\cal D}\sqrt{d})^d$, and $C$ a constant that depends only on $d$, ${\cal D}$, 
  and $\delta$. In particular,
  \begin{align}
    \P(R_{S}(x;\eta_t;\X) \geq r)\leq C\exp(-c_\lambda \kappa tr^{d}) \qm{for all $r > 0$.}\label{eq:Rs.exponential.bound}
  \end{align}
  
\end{proposition}
\begin{proof}
  Let $\pi_d$ be the volume of the $d$-dimensional ball of radius~$1$. 
  First, we show that for any $s>0$ and $0\leq\epsilon<1/2$,
  \begin{align} \label{eq:latticebound}
    \P[s< R_S(B(x,\epsilon s);\mu;\X)\leq 2s] &\leq  \biggl(\frac{2{\cal D}\sqrt{d}}{(1-2\epsilon)\delta}\biggr)^d\pi_d 
        \exp\bigl(-c_\lambda \kappa ts^d\bigr).
  \end{align}
  To prove this claim, suppose that $s < R_S(B(x,\epsilon s);\mu;\X)\leq 2s$ and apply Lemma~\ref{lem:emptyspace}
  to conclude that there exists a ball of radius $(1-2\epsilon)\delta s/{\cal D}$ within $B(x,2s)\cap\X$
  containing no points of $\mu$.
  Now, consider the lattice $((1-2\epsilon)\delta sd^{-1/2}/{\cal D})\mathbb{Z}^d$. 
  By a volume argument, $B(x,2s)\cap\X$ contains at most
  \begin{align*}
    \frac{\vol{B(0,2s)}}{((1-2\epsilon)\delta s d^{-1/2}/{\cal D})^d}=
      \biggl(\frac{2{\cal D}\sqrt{d}}{(1-2\epsilon)\delta}\biggr)^d\pi_d
  \end{align*}
  lattice cells. Any ball of radius $(1-2\epsilon)\delta s/{\cal D}$ contains a cell of this
  lattice. 
  
  In all, we have shown that if $s < R_S(B(x,\epsilon s);\mu;\X)\leq 2s$, then at least one of the 
  at most $(2 {\cal D}\sqrt{d}/(1-2\epsilon)\delta)^d\pi_d$ lattice cells within $B(x,2s)\cap \X$ contains no point of $\mu$.
  With binomial input, applying Assumption \ref{assumption:PoissonDom}, 
  the probability of a single cell being empty is bounded by
 \begin{align} \label{eq:bin.empty.cell}
 \Biggl[ 1- c_\lambda \biggl(\frac{(1-2\epsilon)\delta }{{\cal D}\sqrt{d}}\biggr)^ds^d\Biggr]^{t} 
 &\leq \exp\Biggl[-c_\lambda \biggl(\frac{(1-2\epsilon)\delta}{{\cal D}\sqrt{d}}\biggr)^dts^d\Biggr].
 \end{align} 
 With Poisson input, each lattice cell contains no point of $\mu$ with probability at most the right hand side of 
 \eqref{eq:bin.empty.cell}.
A union bound now proves \eqref{eq:latticebound}. 
  
Now consider $r>0$, arbitrary. If $\exp\bigl(-c_\lambda \kappa t r^d\bigr)>1/2$, then \eqref{eq:Rs.exponential.bound.ball} is trivially satisfied with $C=2$. Otherwise, applying a union bound using \eqref{eq:latticebound} with $s=r,2r,4r,\ldots$ gives
  \begin{align*}
    \P[R_S(B(x,\epsilon r);\mu;\X)>r] &\leq \biggl(\frac{2{\cal D}\sqrt{d}}{(1-2\epsilon)\delta}\biggr)^d\pi_d \sum_{i=0}^{\infty}
       \exp\bigl(-c_\lambda \kappa t (2^ir)^d\bigr).
  \end{align*}
Using $\exp\bigl(-c_\lambda \kappa t r^d\bigr)\leq 1/2$, inequality 
\eqref{eq:Rs.exponential.bound} may now be established
  by bounding the sum in the above inequality
  by a geometric series summing to $2\exp\bigl(-c_\lambda \kappa t r^d\bigr)$.
\end{proof}

\section{Functionals of forbidden regions graphs satisfy a Berry-Esseen bound}
In this section we let $\Ppp{t}$ be a Poisson process with intensity $\lambda_t=t\lambda, t \ge 1$ for some fixed probability measure $\lambda$ on $\X$, and we prove the Poisson input case of Theorem \ref{thm:main}.
For a functional $F_t$ on $\Ppp{t}$ with finite, non-zero variance, {  recall that} \beas
{\widetilde F}_t=(F_t-\E F_t)/\sqrt{{\rm Var}(F_t)}.
\enas

\begin{proposition}[Proposition 1.4, Last, Peccati and Schulte \cite{last2014normal}]\label{prop1.3LPS}
Let $\E F_t^2<\infty, t \ge 1$, and assume there are finite positive constants $p_1,p_2>0$ and $c$ such that
\bea \label{eq1.6LPS}
\E|D_xF_t|^{4+p_1} \le c \qm{$\lambda$-a.e. $x \in \X, t \ge 1$}
\ena
and
\bea \label{eq1.7LPS}
\E|D_{x,y}^2F_t|^{4+p_2} \le c \qm{$\lambda^2$-a.e. $(x,y) \in \X^2, t \ge 1$.}
\ena

Moreover, assume that for some $v>0$
\bea \label{eq:var.lower.bound}
\frac{{\rm Var}(F_t)}{t} \ge v \qm{for all $t \ge 1$,}
\ena
and that
\bea \label{eq1.8LPS}
m:=\sup_{x \in \X, t \ge 1} \int \P (D_{x,y}^2 F_t \not = 0)^{p_2/(16+4p_2)} \lambda_t (dy) <\infty.
\ena
Then there exists a finite constant $C$, depending only on $c,p_1,p_2,v,m$ and $\lambda(\X)$ such that with $d(\cdot,\cdot)$ denoting either the Wasserstein or Kolmogorov distance and $N$ a standard Gaussian random variable,
\beas
d({\widetilde F}_t,N) \le C t^{-1/2} \qm{for all $t \ge 1$.}
\enas
\end{proposition}

We first prove Lemma \ref{lem:DxF.bound.in.alpha}, a bound on the derivative of the functional $L$ in \eqref{eq:defLalpha}, which is used when considering both Poisson and binomial input processes. In preparation, for any finite point configuration $\mu \subseteq \X$ and $x \in \X\setminus \mu$, we let
\begin{multline*} 
A(x;\mu)\\
=\{z \in \mu: \exists w \in \mu, w \not = z,  S(w,z) \cap (\mu \cup \{x\}) = \{x\} \} \bigcup \{z \in \mu: S(x,z) \cap \mu = \emptyset \}.
\end{multline*}

Recalling \eqref{def:calRsubS}, we see
\bea
\label{eq:A-x-mu-included-R-x-mu}
A(x;\mu) \subseteq \bigcup \left\{\overline{S(w,z)}: \{w,z\} \subseteq \X, S(w,z) \cap \mu = \emptyset, x \in  \overline{S(w,z)}\right\} \subseteq \overline{{\cal R}_S(x;\mu;\X)}.
\ena

Let $|A(z;\mu)|$ denote the cardinality of $A(z;\mu)$.
\begin{lemma} \label{lem:DxF.bound.in.alpha}
Let $\mu $ be a locally finite subset of $\mathbb{R}^d$ and $x \in \X$, and let $F=L(\mu)$ where $L(\cdot)$ is given in \eqref{eq:defLalpha} with $|\psi(x,y)| \le C\|x-y\|^\alpha$ for some $\alpha \ge 0, C > 0$. Then there is a constant $C_\alpha$, depending only on $\alpha$ and $C$, such that 
\bea \label{eq:DxFbnd.alphas}
|D_x F| \le C_\alpha \sum_{z \in A(x;\mu)}\|z-x\|^\alpha
\max(|A(z;\mu)|,1).
\ena 
\end{lemma} 
\begin{proof} 
	For $x \in \mu$ we have $D_xF=0$. Otherwise take $x \in \mathbb{X} \setminus \mu$ and,
noting that the insertion of $x$ into $\mu$ can only break existing edges and form new edges incident 
  to $x$, we have  
\beas D_x F
		= -\sum_{\substack{ \{z,w\} \subseteq \mu \\z\neq w, S(z,w) \cap (\mu \cup \{x\}) = \{x\}}}
		\psi(z,w)+ \sum_{z \in \mu, S(z,x) \cap \mu = \emptyset} \psi(z,x).
\enas
	
For the first term we note
	\beas
	|\psi(z,w)| \le C \|z-w\|^\alpha \le C \max(1,2^{\alpha-1}) \left(\|z-x\|^\alpha+\|w-x\|^\alpha \right)
	\enas
	so that
	\begin{multline*} 
	|D_xF| \le C_\alpha \left( \sum_{\substack{ \{z,w\} \subseteq \mu \\z\neq w, S(z,w) \cap (\mu \cup \{x\}) = \{x\}}}
	\|z-x\|^\alpha+\sum_{z \in \mu, S(z,x) \cap \mu = \emptyset} \|z-x\|^\alpha \right) \\
	\le C_{\alpha }\sum_{z\in A(x;\mu)}\|z-x\|^{\alpha}\max(|A(z;\mu)|,1), 
	\end{multline*}
where, for the two sums, we see that if $\{z,w\}$ or $z$, respectively, satisfy the conditions of summation then $z \in A(x;\mu)$, while for the first sum $S(w,z) \cap \mu = \emptyset$, which implies $w \in A(z;\mu)$.
\end{proof}

The proof of the following lemma is provided immediately after the proof of Theorem~\ref{thm:main}; we will make use of the fact that
\bea \label{eq:integral.Gamma.identity}
\int_0^\infty r^\beta \exp(-\gamma r^d) dr = \frac{1}{d\gamma^{(\beta+1)/d}}\Gamma\left(\frac{\beta+1}{d}\right) \qmq{for $\beta>-1, \gamma>0$ and $d >0$.}
\ena

In the following, let $\bpp{t}=\emptyset$ for $t<0$.
	\begin{lemma}
		\label{eq:lemma-moments} 
		For $t\ge 1$ let $\Ppp{t}$ and $\bpp{t}$ be as in Assumption \ref{assumption:PoissonDom}, and let $\mathcal{A}\subseteq \mathbb{R}^d$. Then
		\begin{multline}
		\label{eq:claim-derivative}
		\sup_{t \ge 1,x\in \X,\mathcal{A}\subseteq \X, 0 \le | \mathcal{A} | \leqslant 2} \E |D_{x} t^{\alpha/d}L(\Ppp{t} \cup {\cal A}) | ^{6}<\infty \qmq{and}\\ \sup_{t \ge 1,0 \le k \le 3, x\in \X,\mathcal{A}\subseteq \X,  0 \le | \mathcal{A} | \leqslant 2} \E |D_x t^{\alpha/d} L(\bpp{t-|{\cal A}|-k}\cup {\cal A})|^{6}<\infty.
		\end{multline}
	\end{lemma}

\begin{proof}[Proof of Theorem~\ref{thm:main}, Poisson input] We apply Proposition \ref{prop1.3LPS} to $F_t=t^{\alpha/d}L(\Ppp{t})$, with $L$ as given in \eqref{eq:defLalpha} where $\Ppp{t}$ is a Poisson process satisfying the conditions of Assumption~\ref{assumption:PoissonDom}. First, the condition $\E F_t^2<\infty$
is seen to be satisfied in light of the inequality 
 $|F_{t} |\leq t^{\alpha/d} C \left(\sup_{\{x,y\} \subseteq \X} \|y-x\|\right)^\alpha |\Ppp{t}|^{2}$, where $|\nu|$ denotes the number of points of the process $\nu$.

As Assumption \ref{assumption:sigh} holds by hypothesis, we have
\beas
{\rm Var}(t^{\alpha/d}L(\Ppp{t})) \ge v_\alpha t, 
\enas
verifying \eqref{eq:var.lower.bound}.

Next, choosing $p_1$ and $p_2$ both equal to 1, inequalities \eqref{eq1.6LPS}, \eqref{eq1.7LPS} and \eqref{eq1.8LPS} become, respectively,
\begin{align}
\E  | D_{x}F_{t} |^{5}  \le c, \qmq{$\lambda$-a.e., $x \in \X, t \ge 1$, } \label{EDx.bnded}\\
\E | D^{2}_{x,y} F_{t} |^{5}\le c, \qmq{$\lambda^2$-a.e., $(x,y) \in  \X \times \X, t \ge 1$, } \label{EDxy.bnded}
\end{align}
and
\begin{align} \label{intDxyne0.finite}
\sup_{x\in \X,t\geq 1} t \int_{\X}\P(D^{2}_{x,y}F_{t}\neq 0)^{1/20} \lambda(dy)<\infty .
\end{align}

We next note that by \eqref{eq:A-x-mu-included-R-x-mu},
\bea \label{eq:yinA.implies.Rsbig}
y \in A(x;\mu) \qmq{implies} R_S(x;\mu;\X) \ge \|y-x\|.
\ena

Applying Lemma \ref{eq:lemma-moments} with $\mathcal{A} = \emptyset$ shows that \eqref{EDx.bnded} is satisfied, and letting $\mathcal{A}=\{y\}$ we see that \eqref{EDxy.bnded} also holds, as \eqref{Dxy.explicit} yields
\begin{align*} 
\E  | D^{2}_{x,y}F_{t} |^{5}\leq 16 \left( \E  | D_{x}F_t(\Ppp{t}\cup \{y\}) |^{5}+\E | D_{x}F_t(\Ppp{t}) |^{5}  \right).
\end{align*}

We now show condition \eqref{intDxyne0.finite} is satisfied. Letting $x \in \X$ be arbitrary, invoking Assumption \ref{assumption:PoissonDom} and 
Lemma \ref{lem:stabilization}, followed by Proposition \ref{prop:Rs.exponential.bound} and \eqref{eq:integral.Gamma.identity}, we obtain  
\begin{multline*}
b_\lambda^{-1} t\int_{\X}\P(D^{2}_{x,y}F_{t}\neq 0)^{1/20} \lambda(dy)
\le t\int_{\X}\P(D^{2}_{x,y}F_{t}\neq 0)^{1/20} dy \\
\le t\int_{\X}\P(R_S(x;\Ppp{t};\X) \ge \|y-x\|)^{1/20}dy
\le Ct  \int_{\X}\exp\left(-c_\lambda \kappa t \|y-x\|^d/20\right) dy \\= Ct \int_{\X-x}\exp\left(-c_\lambda \kappa t \|y\|^d/20\right) dy
\le Ct \int_{\mathbb{R}^d}\exp\left(-c_\lambda \kappa t \|y\|^d/20\right) dy \\= Ct \sigma_d \int_0^\infty \exp\left(-c_\lambda \kappa t r^d/20\right) r^{d-1} dr
= \frac{20C \sigma_d}{d c_\lambda \kappa}.
\end{multline*}
Hence, the supremum over $x \in \X$ and $t \ge 1$ in \eqref{intDxyne0.finite} is finite, and the proof of the Poisson input case of  Theorem~\ref{thm:main} is complete. 
\end{proof}

\begin{proof}[Proof of Lemma \ref{eq:lemma-moments}] 
Let $\eta$ denote $\Ppp{t} \cup {\cal A}$ and $\bpp{t-|{\cal A}|-k} \cup {\cal A}$ 
in the Poisson and binomial cases, respectively. With $F_t=t^{\alpha/d}L(\eta)$, for fixed $x \in \mathbb{X}$ we have by Lemma \ref{lem:DxF.bound.in.alpha}
\begin{align}
\label{eq:derivative-as-sum}
 | D_{x}F_t| \leqslant t^{\alpha /d}\sum_{y\in \eta  }\|y-x\|^{\alpha }{\bf 1}(y \in A(x;\eta))(1+|A(y;\eta)|).
\end{align}
We develop a general bound to handle the moments of \eqref{eq:derivative-as-sum}.
Given a positive integer $m\in \mathbb{N}$, we say a set $P=\{m_{1},\dots ,m_{p}\} \subseteq \{1,\ldots,m\}^p$ is a partition of $m$ when  $\sum_{i=1}^{p}m_{i}=m$. Let $\mathcal{P}_{m}$ denote the class of all such partitions. Let $\varphi(z;\eta),z\in \X$ be some non-negative kernel and for a subset $\mu \subseteq \mathbb{R}^d$, let $\mu_{\not =}^p$ denote the collection of all vectors $(z_1,\ldots,z_p)$ with ${\bf z}:=\{z_1,\ldots,z_p\} \subseteq \mu$ and $|{\bf z}|=p$.

By writing the sum \eqref{eq:derivative-as-sum} over $y \in \Ppp{t}$ in the Poisson case, or over $y \in \bpp{t-|{\cal A}|-k}$ for the binomial, plus a sum over $y \in {\cal A}$, and using inequalities of the form $(a+b)^m \le 2^{m-1}(a^m+b^m)$, to obtain a bound on $E|D_xF_t|^m$ it suffices to obtain $m^{th}$ moment bounds on each component summand; see, for instance, \eqref{eq:degree-moment} below.

First consider the Poisson case.  The multivariate Mecke formula as in (2.10) of \cite{LP}, along with the upper bound of Assumption \ref{assumption:PoissonDom} on the intensity of $\Ppp{t}$, yields
\begin{multline}
\E\left(
\sum_{z\in \Ppp{t}}\varphi(z;\eta)
\right)^{m}=\E\sum_{{\bf z}=(z_{1},\dots ,z_{m})\in \Ppp{t}^m }\varphi(z_1;\eta)\cdots \varphi(z_m;\eta)\\
=\sum_{\{m_{1},\dots ,m_{p}\}\in \mathcal{P}_{m}}\E\sum_{{\bf z}=(z_{1},\dots ,z_{p})\in  \Ppp{t,\neq}^{p}}\prod_{i=1}^{p}\varphi(z_i;\eta)^{m_{i}}\\
\leqslant    \sum_{\{m_{1},\dots ,m_{p}\}\in \mathcal{P}_{m}} 
(b_\lambda t)^{p}  \int_{\X^{p}}\E\left[
\prod_{i=1}^{p}\varphi(z_i;\Ppp{t} \cup \mathcal{A} \cup {\bf z})^{m_{i}}
\right]dz_{1}\dots dz_{p}\\
\leqslant  C \sum_{\{m_{1},\dots ,m_{p}\}\in \mathcal{P}_{m}}
t^{p} \int_{\X^{p}}\prod_{i=1}^{p}\left[\E
\varphi(z_i;\Ppp{t}\cup \mathcal{A} \cup {\bf z})^{m}
\right]^{m_{i}/m}dz_{1}\dots dz_{p}. \label{eq:bound-skohokod-integral}
\end{multline}

In the binomial case, a similar computation yields
\begin{multline*}
\E\left(
\sum_{z\in \bpp{t-|{\cal A}|-k}} \varphi(z;\eta)
\right)^{m}=\E\sum_{{\bf z}=(z_{1},\dots ,z_{m})\in  \bpp{t-|{\cal A}|-k}^m}\varphi(z_1;\eta)\dots \varphi(z_m;\eta)\\  
=\sum_{\{m_{1},\dots ,m_{p}\}\in \mathcal{P}_{m}}\E\sum_{{\bf z}=(z_{1},\dots ,z_{p})\in  \bpp{t-|{\cal A}|-k,\neq}^p} \prod_{i=1}^{p}\varphi(z_i;\eta)^{m_{i}}\\
\leqslant   \sum_{\{m_{1},\dots ,m_{p}\}\in \mathcal{P}_{m}}   
p!{t- | \mathcal{A} | -k \choose p}
\int_{\X^{p}}\E\left[
\prod_{i=1}^{p}\varphi(z_{i};\bpp{t-|{\cal A}|-k-p}\cup \mathcal{A}\cup {\bf z})^{m_{i}}
\right]dz_{1}\dots dz_{p}\\
\leqslant C  \sum_{\{m_{1},\dots ,m_{p}\}\in \mathcal{P}_{m}}  t^{p}
\int_{\X^{p}}\prod_{i=1}^{p}\left[\E
\varphi(z_{i};\bpp{t-|{\cal A}|-k-p} \cup \mathcal{A}\cup {\bf z})^{m}
\right]^{m_{i}/m}dz_{1}\dots dz_{p},
\end{multline*}
as in \eqref{eq:bound-skohokod-integral} for the Poisson case.

Returning to the Poisson case, by writing $|A(y;\eta)|$ as a sum over $\Ppp{t}$ added to another over ${\cal A}$, we first control the moments of
\begin{multline}
\label{eq:degree-moment}
\E(1+|A(y;\eta)|)^m = \E\left(1+ \sum_{z\in \eta}\mathbf{1}_{\{z\in A(y;\eta )\}}
\right)^{m} \le 3^{m-1} \left(1+T_1+T_2\right), \qm{where}\\
T_1= \E\left( \sum_{z\in \Ppp{t}}\mathbf{1}_{\{z\in A(y;\eta )\}}
\right)^{m}  
\qmq{and} T_2 = \left( \sum_{z \in {\cal A}}{\bf 1}_{\{z\in A(y;\eta )\}}\right)^m \le |{\cal A}|^m.
\end{multline}
We handle $T_1$ by specializing \eqref{eq:bound-skohokod-integral} to the case where $\varphi(z;\eta)={\bf 1}(z \in A(y;\eta))$, suppressing $y$ for notational ease in the functional.
By \eqref{eq:A-x-mu-included-R-x-mu}, \eqref{def:Rs} and Lemma \ref{lem:monotonicity}, for any collection of points ${\bf z}=\{z_1,\ldots,z_p\}$ we have
\begin{multline*}
z \in A(y;\eta \cup {\bf z}) \implies z \in \overline{{\cal R}_S(y;\eta \cup {\bf z};\mathbb{X})} \implies R_S(y;\Ppp{t} \cup {\cal A} \cup {\bf z};\mathbb{X}) \ge \|y-z\|\\
\implies R_S(y;\Ppp{t} ;\mathbb{X}) \ge \|y-z\|.
\end{multline*}
Proposition \ref{prop:Rs.exponential.bound} now yields that for all ${\mathcal A} \subseteq \mathbb{R}^d$, 
\begin{align}
\label{eq:exp-bound}
\P(z\in A(y;\eta \cup {\bf z}))\leqslant C \exp(-c_\lambda \kappa t\|y-z\|^{d}).
\end{align}
Now, by \eqref{eq:exp-bound} we obtain
\begin{align}
\notag T_1&\leqslant C\sum_{\{m_1,\ldots,m_p\} \in \mathcal{P}_{m}}t^{p}\int_{\X^{p}}\prod_{i=1}^{p}\exp(-c_\lambda \kappa t\|y-z_i\|^{d})dz_{1}\dots dz_{p} \\
\notag  &\leqslant C\sum_{\{m_1,\ldots,m_p\}\in \mathcal{P}_{m}}t^{p}\int_{(\mathbb{R}^{d})^{p}}\prod_{i=1}^{p}\exp(-c_\lambda \kappa t\|y-z_i\|^{d})dz_{1}\dots dz_{p} \\
\notag &=C\sum_{\{m_1,\ldots,m_p\}\in \mathcal{P}_{m}}t^{p}\int_{(\mathbb{R}^{d})^{p}}\prod_{i=1}^{p}\exp(-c_\lambda \kappa t\|z_i\|^{d})dz_{1}\dots dz_{p}\\
\notag & = C \sum_{\{m_1,\ldots,m_p\}\in \mathcal{P}_{m}} \left( t \int_{{\mathbb R}^d}  \exp(-c_\lambda \kappa t\|z\|^{d}) dz \right)^p\\
\notag & = C \sum_{\{m_1,\ldots,m_p\}\in \mathcal{P}_{m}}  \left( t \sigma_d \int_0^\infty  r^{d-1}\exp(-c_\lambda \kappa tr^d) dr \right)^p\\
\notag & = C \sum_{\{m_1,\ldots,m_p\}\in \mathcal{P}_{m}} \left(\frac{\sigma_p}{dc_\lambda \kappa}\right)^p\\
&\leqslant C,\label{eq:bound-skohokod-integral.C} \end{align}
where in the final inequality we apply \eqref{eq:integral.Gamma.identity}, and $C$ depends on $m,\kappa ,c_{\lambda }$. As ${\cal A}$ is finite the term $T_2$ in \eqref{eq:degree-moment} is finite, yielding for all positive integers $m$ a constant $C$ such that
\begin{align}\label{eq:mom.Ayeta.finite}
E(1+|A(y;\eta)|)^m  \le C.
\end{align}
Inequality \eqref{eq:exp-bound}, and then \eqref{eq:bound-skohokod-integral.C} followed by \eqref{eq:mom.Ayeta.finite},
is obtained in the identical manner for the binomial case.

To consider the right hand side of \eqref{eq:derivative-as-sum}, suppressing $x$ for notational ease, let $\varphi (y;\eta )=\|y-x\|^{\alpha }\mathbf{1}_{\{y\in A(x;\eta  )\}} (1+|A(y;\eta)|)$. The Cauchy-Schwarz inequality, the bound \eqref{eq:exp-bound} with ${\bf z}$ any collection of points, and \eqref{eq:mom.Ayeta.finite} with $\eta$ replaced by $\eta \cup {\bf z}$, yield for any $y\in \X$
\begin{multline}
\E |\varphi (y;\eta \cup {\bf z})|^{6} \leqslant    \|y-x\|^{6\alpha  }\sqrt{\P(y\in A(x;\eta \cup {\bf z}))}\sqrt{\E(1+|A(y;\eta \cup {\bf z} )|)^{12}}
\\
\label{eq:moment-phi}
 \leqslant C \|y-x\|^{6\alpha  }\exp(-c_\lambda \kappa t\|y-x\|^{d}/2).
\end{multline}

Now decompose the right hand side of \eqref{eq:derivative-as-sum} into two summands as in \eqref{eq:degree-moment}; we only consider the Poisson case, the binomial case being identical after replacing $\Ppp{t}$ with $\bpp{t- | \mathcal{A} | -k}$. For the sum over $\Ppp{t}$, using \eqref{eq:bound-skohokod-integral} and \eqref{eq:moment-phi}, we obtain 
\begin{align}
\notag &\E\left| t^{\alpha /d}\sum_{y\in\Ppp{t} }\varphi (y;\eta)\right|^{6}\\
\notag &\leqslant C t^{6\alpha /d}\sum_{\{m_1,\ldots,m_p\}\in \mathcal{P}_{6}}t^{p}\int_{(\mathbb{R}^{d})^{p}}\prod_{i=1}^{p}\left(
\|y_{i}-x\|^{6 \alpha }\exp(-c_\lambda \kappa  t\|y_{i}-x\|^{d}/2)
\right)^{m_{i}/6}dy_{1}\dots dy_{p}\\
\notag &\leqslant C t^{6\alpha /d}\sum_{\{m_1,\ldots,m_p\}\in \mathcal{P}_{6}}t^{p}\int_{(\mathbb{R}^{d})^{p}}\prod_{i=1}^{p} 
( \| y_{i}\|^{6 \alpha }\exp(-c_\lambda \kappa t\| y_{i}\|^{d}/2))^{m_i/6}
dy_{1}\dots dy_{p}\\
\notag &  = C t^{6\alpha /d}\sum_{\{m_1,\ldots,m_p\}\in \mathcal{P}_{6}}t^{p} \prod_{i=1}^{p}   \int_{\mathbb{R}^{d}}
 \| y\|^{\alpha m_i}\exp(-m_i c_\lambda \kappa t\| y\|^{d}/12)
dy \\
\notag &  =C t^{6\alpha /d}\sum_{\{m_1,\ldots,m_p\}\in \mathcal{P}_{6}}t^p \prod_{i=1}^{p}  \sigma_d  \int_0^\infty 
r^{\alpha m_i+d-1}\exp(-m_i c_\lambda \kappa tr^{d}/12)
dr \\
\notag &  = C t^{6\alpha /d}\sum_{\{m_1,\ldots,m_p\}\in \mathcal{P}_{6}}t^p\prod_{i=1}^{p}   
\frac{\sigma_d}{d(m_i c_\lambda \kappa t/12)^{\alpha m_i/d+1}}\Gamma\left(\frac{\alpha m_i}{d}+1\right)\\
&  =C \sum_{\{m_1,\ldots,m_p\}\in \mathcal{P}_{6}}  \prod_{i=1}^{p}   
\frac{\sigma_d^{p}}{d(m_i c_\lambda \kappa/12)^{\alpha m_i/d+1}}\Gamma\left(\frac{\alpha m_i}{d}+1\right), \label{eq:Pt.term}
\end{align}
where we have used \eqref{eq:integral.Gamma.identity}
with $\beta=\alpha m_i + d-1$ and $\gamma= m_ic_\lambda \kappa t/12$ in the next to last equality.

Now considering the sum over ${\cal A}$, and setting $v=t\|y-x\|^d$ in the last inequality, we obtain
\begin{align}
\notag \E\left|t^{\alpha/d}\sum_{y\in \mathcal{A}}\varphi(y;\eta)
\right|^{6}  & \leqslant |\mathcal{A}|^{6} t^{6\alpha/d}\E \varphi(\eta;y)\\
\notag  &\leqslant C|\mathcal{A}|^{6} t^{6\alpha/d} \|y-x\|^{6\alpha }\exp(-c_\lambda \kappa t\|y-x\|^{d}/2) \\
&\leqslant C | \mathcal{A}|^{6}\left(
\sup_{v>0 }v^{ 6\alpha /d}  \exp(-c_\lambda \kappa v/2)
\right). \label{eq:A.term}
\end{align}

As \eqref{eq:Pt.term} and \eqref{eq:A.term} are constants not depending on $t$ or $x \in \mathbb{X}$, the proof is complete. 
\end{proof}

We shall now use the results of \cite{LRP} to prove Theorem~\ref{thm:main} for binomial input.
Here $n\in \mathbb{N}$ plays the former role of $t$ and $X=(X_{1},\dots ,X_{n})$ is a vector of independent variables with distribution $\lambda$ over $\X$, and $\bpp{n}=\{X_{1},\dots ,X_{n}\}$. Let $X',\widetilde{X}$ be independent copies of $X$. We write $U\stackrel{a.s.}{=}V$ if two variables $U$ and $V$ satisfy $\P(U=V)=1$. In the vocabulary of \cite{LRP}, a random vector $Y=(Y_{1},\dots ,Y_{n})$ is a recombination of $\{X,X',\widetilde X\}$ if for each $1\leqslant i\leqslant n$, either { $Y_{i}\stackrel{a.s.}{=}X_{i},Y_{i}\stackrel{a.s.}{=}Y'_{i}$ or $Y_{i}\stackrel{a.s.}{=}\widetilde X_{i}$}. 
For a vector $x=(x_{1},\dots ,x_{n})$, and indices $\{i_{1},\dots ,i_{q}\} \subseteq \{1,\ldots,n\}$, define 
\begin{align*}
x^{i_{1},\dots ,i_{q}}: =(x_{j},j\notin \{i_{1},\dots ,i_{q}\}).
\end{align*}
For $1\leq i,j\leq n$, and $f$ a real valued function taking in $n, n-1$ or $n-2$ ordered arguments in $\mathbb{R}^d$, let
 \begin{align} \label{def:mathsfD}
\mathsf{D}_{i}f(X)&=f(X)-f(X^i) \qm{and}\\
\mathsf{D}_{i,j}f(X)&=f(X)-f(X^i)-f(X^j)+f(X^{i,j}), \qmq{noting that $\mathsf{D}_{i,j}f(X)=\mathsf{D}_{j,i}f(X)$.}\nonumber
\end{align}

Recalling that $X',{\widetilde X}$ are independent copies of $X$, let
\begin{align*}
B_{n}(f)&=\sup\{\gamma _{Y,Z}(f):\;{(Y,Z)\text{ recombinations of }\{X,X',\widetilde X\}}\} \qm{and}\\
B'_{n}(f)&=\sup\{\gamma'_{Y,Y',Z}(f):\;{(Y,Y',Z)\text{ recombinations of }\{X,X',\widetilde X\}}\}, \qm{where}\\
\gamma _{Y,Z}(f)&=\E \left[ \mathbf{1}_{\{\mathsf{D}_{1,2}f(Y)\neq 0\}}\mathsf{D}_{2}f(Z)^{4} \right] \qm{and}\\
\gamma'_{Y,Y',Z}(f)&=\E\left[ \mathbf{1}_{\{\mathsf{D}_{1,2}f(Y)\neq 0,\;\mathsf{D}_{1,3}f(Y')\neq 0\}} \mathsf{D}_{2}f(Z)^{4}\right].
\end{align*}

{  Then Theorem 5.1 of \cite{LRP}, simplified by \cite[Remark 5.2]{LRP} and \cite[Proposition 5.3]{LRP} yields the following Kolmogorov distance bound for the normal approximation of $f(X)$, properly standardized.}

 \begin{theorem}[Lachi\`eze-Rey and Peccati \cite{LRP}] \label{thm:LRP}
Let $f$ be a  functional taking in ordered arguments of $n,n-1$, or $n-2$ elements of $\X$.   Assume furthermore that $f$ is invariant under permutation of its arguments, that $\E f(X)=0$ and that $\sigma ^{2}:=\var(f(X))$ is non-zero and finite. Let $d(\cdot,\cdot)$ denote either the Kolmogorov or the Wasserstein distance. Then, for some $C>0$ not depending on  {$f$ or }$n$,
\begin{align} 
 d (\sigma ^{-1}f(X),N)& \leqslant C\left[\frac{{  4}\sqrt{2}  n^{1/2}}{\sigma^2} \left( \sqrt{nB_{n}(f)  } +\sqrt{n^{2}B'_{n}(f)}+ \sqrt{\E \mathsf{D}_{1}f (X)^{4}}\right) \right. \label{bd.lrp} \\
\notag &\left. \hspace{2cm}+ { \frac{n}{4\sigma^3}}  \sqrt{\E  |  {    \mathsf{D}_{1}f(X) ^{6} }  | }+  \frac{ \sqrt{{2} \pi } n}{16\sigma ^{3}}\E  |   \mathsf{D}_{1}f(X)^{3} |    \right], \nonumber
 \end{align}
 where $N$ is a standard normal random variable.
\end{theorem}

The authors of \cite{LRP} focus on the Kolmogorov distance, but the bound they prove is valid for the Wasserstein, even though it is not stated there formally. More precisely, we refer the reader to the inequality in Theorem 2.2 of \cite{Cha08}, involving Wasserstein distance. The first term in this inequality, $\sigma ^{-2}\sqrt{\var(\E(T | W))}$, has been shown in \cite{LRP} to be bounded by the terms of the first line of the right hand member of \eqref{bd.lrp}. The second term in the inequality of \cite{Cha08} is equal to $n\sigma ^{-3}\E | \mathsf{D}_{1}f(X) | ^{3}$, also taken care of in \eqref{bd.lrp}. The term $ (n/(4\sigma^3)) \sqrt{\E  |  {    \mathsf{D}_{1}f(X) ^{6} }  | }$ is in fact only necessary for the Kolmogorov distance, and can be removed when treating the Wasserstein distance. Hence the upper bound \eqref{bd.lrp} for the Kolmogorov distance also upper bounds the Wasserstein.

For $L$ as in \eqref{eq:defLalpha} with $|\psi(x,y)| \le C\|x-y\|^\alpha$ for some $\alpha \ge 0, C > 0$, let $F_n=n^{\alpha /d}L(\bpp{n})$, and let the functional $f$, defined on ordered sets of variables,  be given by $f(x_{1},\dots ,x_{q})=F_n(\{x_{1},\dots ,x_{q}\}){ -\E f(X)}$ for any $q\geqslant 1$ and $\{x_1,\ldots,x_q\} \subseteq \mathbb{R}^{d}$. We note that 
$\mathsf{D}$ defined in \eqref{def:mathsfD}, and $D$ as in \eqref{def:Dsubx}, obey the relations  
\bea \label{eq:D.and.D}
\mathsf{D}_{i}f(X) = D_{X_i}F_n(\bpp{n} \setminus \{X_i\}), \qmq{and for $i \not =j$} \mathsf{D}_{ij}f(X) = D_{X_i,X_j}F_n(\bpp{n} \setminus \{X_i,X_j\}).
\ena

We now show how Theorem \ref{thm:LRP},  and Lemma \ref{lem:BnBnprime} below, prove the Kolmogorov and Wasserstein bounds of Theorem \ref{thm:main} for binomial input. 

\begin{proof}[Proof of Theorem~\ref{thm:main} for binomial input]
Assumption \ref{assumption:sigh} yields $\sigma ^{2}\geqslant Cn$ for some $C>0$. Using \eqref{eq:D.and.D} and  \eqref{eq:claim-derivative} of Lemma \ref{eq:lemma-moments} with $\mathcal{A}=\emptyset,k=1$ we obtain,
\begin{align*}
\sup_{n \ge 1}\E [\mathsf{D}_{1}f(X)^{6}]=\sup_{n \ge 1}\int_{\X}\E |D_{x} F_n(\bpp{n-1})|^6 \lambda(dx) <\infty .
\end{align*} 
For the last three terms of \eqref{bd.lrp}, applying H\"older's inequality, we find that there exists $C>0$ such that
\begin{align*}
\frac{  4\sqrt{2}  n^{1/2}}{\sigma ^{2}}  \sqrt{\E \mathsf{D}_{1}f (X)^{4}}  + { \frac{n}{4\sigma^3}}  \sqrt{\E  |  {\mathsf{D}_{1}f(X) ^{6} }  | }+  \frac{ \sqrt{{2} \pi } n}{16\sigma ^{3}}\E  |   \mathsf{D}_{1}f(X)^{3} |   \leqslant Cn^{-1/2}.
\end{align*}
Lemma \ref{lem:BnBnprime} below yields $C$ such that
\bea \label{eq:Bn.Bnprime.inequalities}
B_{n}(f)\leqslant \frac{C}{n} \qmq{and}
B'_{n}(f) \leqslant \frac{C}{n^{2}}.
\ena
Applying these bounds for the first two terms in \eqref{bd.lrp} completes the proof.
\end{proof}

 \begin{lemma} \label{lem:BnBnprime}
There exists $C$ such that
\beas 
B_{n}(f)\leqslant \frac{C}{n} \qmq{and}
B'_{n}(f) \leqslant \frac{C}{n^{2}}. 
\enas
\end{lemma}

\begin{proof} We begin with the first inequality.
Let $Y=(Y_{1},\dots ,Y_{n})$ and $Z=(Z_{1},\dots ,Z_{n})$ be recombinations of $\{X,X',\widetilde X\}$. Note that $Y_{1}$ is independent of $\{Y_{2},Z_{2}\}$ because $Y_1$ is either $X_1,X_1'$ or $\widetilde X_1$ and these three variables are independent of $X_{2},X'_{2},\widetilde X_{2}$. Also, either $Y_{2},Z_{2}$ both equal the same element of $\{X_{2},X'_{2},\widetilde X_{2}\}$, in which case  $Y_{2}\stackrel{a.s.}{=}Z_{2}$ , or they are assigned to different elements of this set, in which case they are independent.
Letting $\lambda ^{Y_{1},Y_{2},Z_{2}}$ denote the law of $(Y_{1},Y_{2},Z_{2})$, we therefore have
$d\lambda ^{Y_{1},Y_{2},Z_{2}}(y_{1},y_{2},z_{2})=\mathbf{1}_{\{y_{2}=z_{2}\}}d\lambda (y_{1})d\lambda (y_{2})$ in the first case, and $d\lambda ^{Y_{1},Y_{2},Z_{2}}(y_{1},y_{2},z_{2})=d\lambda (y_{1})d\lambda (y_{2})d\lambda (z_{2})$ in the second.

Using the conditional H\"older inequality with conjugate exponents $3,3/2$ yields that for every $\{y_{1},y_{2},z_{2}\} \subseteq \X$, with the following conditionings valid $\lambda ^{(Y_{1},Y_{2},Z_{2})}-$a.s., 
\begin{multline} \label{eq:integrate.for.gamma}
\E
\left[
\mathbf{1}_{\{\mathsf{D}_{1,2}f(Y)\neq 0\}}  \mathsf{D}_{2}f(Z)^{4} | Y_{1}=y_{1},Y_{2}=y_{2},Z_{2}=z_{2}
\right]\\
\leqslant \P(\mathsf{D}_{1,2}f(Y)\neq 0 |  Y_{1}=y_{1},Y_{2}=y_{2},Z_{2}=z_{2}
)^{1/3}
\\ \hspace{2cm}\times  \E
\left[
\mathsf{D}_{2}f(Z)^6 |  Y_{1}=y_{1},Y_{2}=y_{2},Z_{2}=z_{2}
\right]^{2/3}.
\end{multline} 
Either $Z_{2}\stackrel{a.s.}{=}Y_{2}$, and when conditioning on $Y_{2}=y_{2},Z_{2}=z_{2}$ we must take $y_{2}=z_{2}$, or $Y_{2}$ and $Z_{2}$ are independent. In both cases, for $\lambda ^{Y_{1},Y_{2},Z_{2}}$-a.e.  $(y_{1},y_{2},z_{2})$, with ${\cal L}(U)$ denoting the law of $U$, and adopting similar notation for the conditional law, by \eqref{eq:D.and.D} we have
\begin{align*}
{\cal L}\left(
  \mathsf{D}_{1,2}f(Y) | Y_{1}=y_{1},Y_{2}=y_{2},Z_{2}=z_{2}
 \right)= {\cal L}(D_{y_{1},y_{2}}F_{n}(\bpp{n-2})).
\end{align*}

Similarly, separately studying the cases $Y_1\stackrel{a.s.}{=}Z_1$ and $(Y_1,Z_1)$ independent, one has for $\lambda ^{Y_{1},Y_{2},Z_{2}}$-a.e.$ (y_{1},y_{2},z_{2}),$ 
\begin{align*}
{\cal L}\left(\mathsf{D}_{{2}}f(Z) | Y_{1}=y_{1},Y_{2}=y_{2},Z_{2}=z_{2} \right)=
\left\{ \begin{array}{ll}
{\cal L}(D_{z_{2}}F_n(\bpp{n-2} \cup \{y_1\})) &  \mbox{if $Y_{1}\stackrel{a.s.}{=}Z_{1}$}\\
{\cal L}(D_{z_{2}}F_n(\bpp{n-1})) &  \mbox{if $Y_{1},Z_{1}$ are independent.}
\end{array}
\right.
\end{align*}

Applying \eqref{eq:claim-derivative} of Lemma \ref{eq:lemma-moments}  with $x=z_2, \mathcal{A}=\{y_1\}$ and $k=2$ for the first case above, and similarly for the second, shows the final factor in  \eqref{eq:integrate.for.gamma} is bounded by $M$. Now integrating \eqref{eq:integrate.for.gamma}
over  $\lambda^{Y_1,Y_2,Z_2}$ and applying Lemma \ref{lem:stabilization} and Proposition \ref{prop:Rs.exponential.bound} yields
\begin{align*} 
\gamma _{Y,Z}(f)
&\leqslant C \int_{\X^{2}}\P(D_{y_{1},y_{2}}F_n(\bpp{n-2})\neq 0)^{1/3}dy_{1}dy_{2} \\
&\leqslant C \int_{\X^2}\P(R_S(y_1;\bpp{n-2};\X) \ge \|y_2-y_1\|)^{1/3}dy_1 dy_2\nonumber \\
&\leqslant C \int_{\X^{2}}C\exp(-c_{\lambda }\kappa (n-2)\|y_{1}-y_{2}\|^{d}/3)dy_{1}dy_{2}\leqslant \frac{C}{n} \nonumber
\end{align*}
for some final constant $C>0$, demonstrating the first inequality in \eqref{eq:Bn.Bnprime.inequalities}.

The second inequality in \eqref{eq:Bn.Bnprime.inequalities}
is proved similarly. Let $Y,Y',Z$ be recombinations of $\{X,X',\widetilde{X}\}$.
Applying the conditional H\"older inequality for a three way product,

\begin{align*}
\gamma '_{Y,Y',Z}(f)&\leqslant \int_{\X^{5}} \P(\mathsf{D}_{1,2}f(Y)\neq 0 | Y_{1}=y_{1},Y_{2}=y_{2},Y'_{1}=y'_{1},Y'_{3}=y'_{3},Z_{2}=z_{2})^{1/6}\\
&\hspace{2cm}\P(\mathsf{D}_{1,3}f(Y')\neq 0 |Y_{1}=y_{1},Y_{2}=y_{2},Y'_{1}=y'_{1},Y'_{3}=y'_{3}, Z_{2}=z_{2} )^{1/6}\\
&\hspace{3cm}\E\left[\mathsf{D}_{{2}}f(Z)^6 | Y_{1}=y_{1} ,Y_{2}=y_{2},Y'_{1}=y'_{1},Y'_{3}=y'_{3},Z_{2}=z_{2} \right]^{2/3}\\
&\hspace{6cm}d\lambda ^{Y_{1},Y_{2},Y'_{1},Y'_{3},Z_{2}}(y_{1},y_{2},y'_{1},y'_{3},z_{2}),
\end{align*}
with the conditionings valid $\lambda ^{Y_{1},Y_{2},Y'_{1},Y'_{3},Z_{2}}$-a.e.
We have, for some $m\in \{0,1,2\}$ and $\mathcal{A} \subseteq \X$ with $|\mathcal{A}|=m$, depending on how the recombination $Z$ is composed,
\begin{align*}
{\cal L}\left( \mathsf{D}_{{2}}f(Z) | Y_{1}=y_{1},Y_{2}=y_{2},Y'_{1}=y'_{1},Y'_{3}=y'_{3},Z_{2}=z_{2} \right) =  {\cal L}(D_{z_{2}}F_n(\bpp{n-1-m}\cup \mathcal{A})),
\end{align*}whenever $Y_2 \stackrel{a.s.}{=} Z_2$, necessitating $y_{2}=z_{2}$, or $Y_{2},Z_{2}$ are independent. Hence, \eqref{eq:claim-derivative} of Lemma \ref{eq:lemma-moments} yields that the last term in the integral is a.e. bounded by $M^{2/3}$.

The values of $Y'_{1},Z_{2}$ are irrelevant to $Y$ once we have conditioned on the values of $Y_{1},Y_{2}.$ Therefore we have
\begin{align*}
\P(\mathsf{D}_{1,2}f(Y )\neq 0 | Y_{1}=y_{1},Y_{2}&=y_{2},Y'_{1}=y'_{1},Y'_{3}=y'_{3},Z_{2}=z_{2})\\
&=\begin{cases} 
\P(D_{y_{1},y_{2}}F_n(\bpp{n-2})\neq 0)\text{ if $Y_{3}$ is independent of $Y'_{3}$}\\
\P(D_{y_{1},y_{2}}F_{n}(\bpp{n-3}\cup \{y'_{3}\})\neq 0)\text{ if }Y_{3}\stackrel{a.s.}{=}Y'_{3}
\end{cases} \\
&\leqslant \P(R_{S}(y_{1},\bpp{n-3})\geqslant \|y_{1}-y_{2}\|) \leqslant C\exp(-c_{\lambda }\kappa (n-3)\|y_{1}-y_{2}\|^{d}),
\end{align*}
where we have used that $R_S(x;\mu;\X)$ stabilizes, from Lemma~\ref{lem:stabilization}, and that
$$\max(R_{S}(y_{1};\bpp{n-2};\X),R_{S}(y_{1};\bpp{n-3}\cup \{y'_{3}\};\X))\leqslant R_{S}(y_{1};\bpp{n-3};\X),$$
justified by the monotonicity property provided by Lemma~\ref{lem:monotonicity}, 
 and Proposition \ref{prop:Rs.exponential.bound}.
Similarly, as the value of $Y_1$ is irrelevant to $Y'$ once we condition on $Y_1'$, and $Y_2'$ will either equal one of $Y_2$ or $Z_2$ a.s., or be independent of both,
for some $m\in \{0,1\}$ and some set $\mathcal{A} $ with $m$ elements, 
\begin{align*}
\P(\mathsf{D}_{1,3}&f(Y')\neq 0 | Y_1=y_1,Y_{2}=y_{2},Y'_{1}=y'_{1},Y'_{3}=y'_{3},Z_{2}=z_{2})\\
&\leqslant 
\P(R_{S}(y'_{1},\bpp{n-2-m}\cup \mathcal{A} )\geqslant \|y'_{1}-y'_{3}\|)\\
&\leqslant C\exp(-c_{\lambda }\kappa (n-3)\|y'_{1}-y'_{3}\|^{d}).
\end{align*}

If $Y_1\stackrel{a.s.}{=}Y_1'$ and $n \ge 4$ we have 
\begin{align*}
\gamma '_{Y,Y',Z}(f)&\leqslant C\int_{\X}\left[ \int_{\X}\exp(-c_{\lambda }\kappa (n-3)\|y_{1}-y_{2}\|^{d}/6)dy_{2} \right]\left[ \int_{\X}\exp(-c_{\lambda }\kappa (n-3)\|y_{1}-y'_{3}\|^{d}/6) dy'_{3}\right]dy_{1}\\
&{  \leqslant C\int_{\X}\left[ \int_{\mathbb{R}^{d}}\exp(-c_{\lambda }\kappa (n-3)\|y_{1}-y_{2}\|^{d}/6)dy_{2} \right]^{2}dy_{1}}\\
&{  = C\int_{\X}\left[ \int_{\mathbb{R}^{d}}\exp(-c_{\lambda }\kappa (n-3)\| y_{2}\|^{d}/6)dy_{2} \right]^{2}dy_{1}}\\
&\leqslant C\int_{\X}\left[ (n-3)^{-1}\int_{\mathbb{R}^d}\exp(-c_{\lambda }\kappa \|y_{1}-y_{2}\|^{d}/6)dy_{2} \right]^{2}dy_{1}\\
&\leqslant \frac{C}{n^2}.
\end{align*}If $Y_{1}$ and $Y'_{1}$ are independent,
\begin{multline*}
 \gamma '_{Y,Y',Z}(f)\\
 \leqslant  C\int_{\X^{2}}\exp(-c_{\lambda }\kappa (n-3)\|y_{1}-y_{2}\|^{d}/6)dy_{1}dy_{2}\int_{\X^{2}}\exp(-c_{\lambda }\kappa(n-3) \|y'_{1}-y'_{3}\|^{d}/6)dy'_{1}dy'_{3}\\
=C\left[  \int_{\X}\left[ \int_{\X}\exp(-c_{\lambda }\kappa (n-3)\|y_{1}-y_{2}\|^{d}/6)dy_{1} \right]dy_{2}\right]^{2}\\
\leqslant C\left[  \int_{\X}\left[ (n-3)^{-1}\int_{\mathbb{R}^{d}}\exp(-c_{\lambda }\kappa \| y_{2}\|^{d}/6)dy_{1} \right]dy_{2}\right]^{2}
\leqslant \frac{C}{n^{2}}.
  \end{multline*}
In both cases, $B'_{n}(f)\leqslant C/n^{2}$, which concludes the proof.
\end{proof}
 
\section{Variance lower bounds}\label{sec:vlb}

In this section, we prove Theorem~\ref{prop:variance}, providing a lower bound on $\var L(\eta_t)$ under broad conditions on the collection of forbidden regions. One key step of the proof, accomplished in Lemma~\ref{lem:influential}, is to show that
if the input process is split into two independent processes then the first process is likely to contain many \emph{influential} point pairs. Intuitively, a point pair $(x,y) \in \mathbb{R}^d \times \mathbb{R}^d$ is influential if an additional process point falling in the vicinity of $x$ produces an effect on $L$ that differs from its effect had the point fallen in the vicinity of $y$. To prove Theorem~\ref{prop:variance}, 
we show that conditional on the first process containing many influential pairs, the effect of adding the second process contributes at least an amount $\Omega(t)$, a quantity satisfying $\liminf_{t \rightarrow \infty}\Omega(t)/t>0$, to the variance of $L(\eta_t)$.

Throughout this section we assume that the function $\psi$ used to define $L$ in \eqref{eq:defLalpha} satisfies the hypotheses of Theorem \ref{prop:variance}. In addition, we will be working at a different scale from the rest of the paper, considering Poisson and binomial processes of constant intensity on a growing space,  rather than of growing intensity on a fixed space. The reason for using this scaling is that we will need to consider the limiting case of a Poisson process on $\mathbb{R}^d$. In particular, in this section, for any $t\geq 1$, we let $\Ppp{t}$ denote
a homogeneous Poisson point process on $t^{1/d}\X$ with intensity $1$, and let $\bpp{t}$ denote a binomial process of $\lceil t\rceil$ points independently and uniformly placed in $t^{1/d} \X$. We couple all $\Ppp{t}$ by defining $\Ppp{t}=\Ppp{\infty}\cap t^{1/d}\X$ where
$\Ppp{\infty}$ is a homogeneous Poisson point process on $\RR^d$ of intensity $1$. 
 
We assume throughout that $\X$ is star shaped with star center at the origin, and contains an open set around the origin. The first property implies that $s^{1/d}\X\subseteq t^{1/d}\X$
	if $s\leq t$, and the second that for all $x \in \mathbb{R}^d$ and $r>0$ that there exists a finite value $t_0(x,r)$ such that
\begin{align}
  B(x,r)\subseteq t^{1/d}\X\qquad\quad\quad\text{for all $t>t_0(x,r)$.}\label{eq:growth}
\end{align}

Before stating the following result we recall the definition of $E_x^{\pm}(\mu)$ from \eqref{def:Epm}, and inform the reader that the constant $r_0$ may take on different values in the statements below. 
\begin{proposition}\label{prop:Epm}
    Assume that the forbidden regions satisfy the scaled ball condition (Assumption \ref{assumption:scaledball}) for some { fixed} $\delta>0$ and all $x \in \mathbb{R}^d$ and positive $t,r$ when the 
    role of $\X$ is played by $t^{1/d}\X\cap B(x,r)$.
    Then for any $\epsilon>0$, there exists $r_0$ such that for all $r>r_0$, all
    $x\in \RR^d$ and all $t \in (t_0(x,r),\infty]$,
    \begin{align}
      \P\Bigl(   \label{eq:EpmPoisson}
        E^{\pm}_x(\Ppp{t}) = E^{\pm}_x(\Ppp{\infty}\cap B(x,r))\Bigr) &\geq 1-\epsilon,\\
     \intertext{and for all $t \in (t_0(x,r),\infty)$,}
      \P\Bigl( E^{\pm}_x(\bpp{t}) = E^{\pm}_x(\bpp{t}\cap B(x,r)) \Bigr)
       &\geq 1-\epsilon.  
    \end{align}
\end{proposition}
Before proving Proposition \ref{prop:Epm}, first observe that \eqref{eq:EpmPoisson} could be equivalently stated with $\Ppp{t}$ appearing
    instead of $\Ppp{\infty}$, since if $B(x,r)\subseteq t^{1/d}\X$, then
    $\Ppp{t}\cap B(x,r)=\Ppp{\infty}\cap B(x,r)$.
    
For $x\in \mathbb{R}^{d}$, $r>0, t > t_0(x,r)$ and a point process $\mu$, define the events
    \begin{align*}
      \Phi(x,r,t,\mu)&=\bigl\{\mathcal{R}_S\bigl(x;\mu;t^{1/d}\X\bigr)\cap 
      \mathcal{R}_S\bigl(x;\mu\cap B(x,r); B(x,r)\bigr)^c\neq\emptyset\bigr\}\\\intertext{and}
      \Psi(x,r,t,\mu) &= \bigl\{\mathcal{R}_S\bigl(x;\mu\cap B(x,r); B(x,r)\bigr)
        \cap\mathcal{R}_S\bigl(x;\mu;t^{1/d}\X\bigr)^c\neq\emptyset\bigr\}.
    \end{align*}
    Note that since $t>t_0(x,r)$, we have $B(x,r)\subseteq t^{1/d}\X$.
    Thus, to picture these events, start with the point process restricted to the viewing
    window $B(x,r)$, and consider the region $\mathcal{R}_S\bigl(x;\mu\cap B(x,r); B(x,r)\bigr)$ 
    that is affected by the addition of $x$ to $\mu$. The first event is that this affected region grows when 
    we expand the window to $t^{1/d}\X$, and the second event is that it shrinks.
    To prove Proposition~\ref{prop:Epm} we require the following result showing that these
    events are unlikely.

    \begin{lemma} \label{lem:bd.OmPs}
      Under the hypotheses of Proposition \ref{prop:Epm}, given any $\epsilon>0$, there exists $r_0$ such that for all $r > r_0$  and $x\in \mathbb{R}^{d}$
       \begin{align*}
       \P\bigl(\Phi(x,r,t,\Ppp{t})\bigr)<\epsilon/2 \qmq{and}
       \P\bigl(\Psi(x,r,t,\Ppp{t})\bigr)<\epsilon/2 \qmq{for $t\in (t_0(x,r),\infty]$,}
       \end{align*}
       and
              \begin{align*}
              \P\bigl(\Phi(x,r,t,\bpp{t})\bigr)<\epsilon/2 \qmq{and}
              \P\bigl(\Psi(x,r,t,\bpp{t})\bigr)<\epsilon/2 \qmq{for $t\in (t_0(x,r),\infty)$.}
              \end{align*}
    \end{lemma}
    \begin{proof}
      We use the same argument as in Proposition~\ref{prop:Rs.exponential.bound}.
      Suppose that $\Phi(x,r,t,\mu)$ holds for $\mu=\Ppp{t}$ or $\mu=\bpp{t}$. 
      Then there exist points $\{w,z\}$ such that
      \begin{enumerate}\renewcommand{\theenumi}{\alph{enumi}}\renewcommand*\labelenumi{(\theenumi)}
      	\item $\{w,z\}\subseteq t^{1/d}\X$;\label{i0}
        \item $S(w,z)\cap \mu = \emptyset$; \label{i1}
        \item $x\in\overline{S(w,z)}$;      \label{i2}
        \item $S(w,z)\not\nsubseteq \mathcal{R}_S\bigl(x; \mu\cap B(x,r); B(x,r)\bigr)$. \label{i3}
      \end{enumerate}
      If $\{w,z\}\subseteq B(x,r)$, then (\ref{i3}) is a contradiction. Thus either 
      $\norm{w-x}>r$ or $\norm{z-x}>r$.
      For $u>0$, let $\widehat{\Phi}(u)$ be the event that there exists $\{w,z\}$ such that 
      (\ref{i0})--(\ref{i2}) hold and
      \begin{align*}
        u<\max\bigl(\norm{w-x},\norm{z-x}\bigr)\leq 2u.
      \end{align*}
      We have now shown that if $\Phi(x,r,t,\mu)$ holds, then
      there exist points $\{w,z\}$ such that (\ref{i0})--(\ref{i2}) hold
      and $\max(\norm{w-x},\norm{z-x})>r$, implying that
      \begin{align} \label{eq:Eunion}
      \Phi(x,r,t,\mu)\subseteq \bigcup_{i=0}^{\infty}\widehat{\Phi}(2^ir).  
      \end{align}
    
      For a given $u>0$ we bound the probability of $\widehat{\Phi}(u)$ and apply a union bound.
      If $\widehat{\Phi}(u)$ holds, then $\{w,z\}\subseteq t^{1/d}\X\cap B(x,2u)$, and $S(w,z)$
      contains no points of $\mu$ and has diameter at least $u$. By the scaled ball condition,
      with the role of $\X$ played by $t^{1/d}\X\cap B(x,2u)$, 
      the set $S(w,z)\cap t^{1/d}\X\cap B(x,2u)$ contains a ball of radius $\delta u/\mathcal{D}$.
      Thus, $\widehat{\Phi}(u)$ implies the existence of a ball of radius $\delta u/\mathcal{D}$
      within $t^{1/d}\X\cap B(x,2u)$ containing no points of $\mu$. Every ball of radius
      $\delta u/\mathcal{D}$ contains a cell of the lattice $(\delta u/\mathcal{D} \sqrt{d})\mathbb{Z}^d$,
      and by considering the volume of $B(x,2u)$, the set $t^{1/d} \X \cap B(x,2u)$ contains at most
      \begin{align*}
        \frac{\pi_d (2u)^d}{(\delta u/\mathcal{D} \sqrt{d})^d} = \frac{\pi_d(2\mathcal{D}\sqrt{d})^d}{\delta^d}
      \end{align*}
      cells of this lattice. 
      Bounding $\widehat{\Phi}(u)$ by the event that all of these cells have no points of $\mu$,
      in the case $\mu=\Ppp{t}$, recalling that $\Ppp{t}$ has intensity 1, 
      \begin{align*}
        \P(\widehat{\Phi}(u)) &\leq \frac{\pi_d(2\mathcal{D}\sqrt{d})^d}{\delta^d}\exp\left(- 
          \kappa u^d\right),
      \end{align*}
      where $\kappa = (\delta/\mathcal{D} \sqrt{d})^d$
      If $\mu=\bpp{t}$, a similar statement holds, as
      \begin{align*}
        \P(\widehat{\Phi}(u)) &\leq  \frac{\pi_d(2\mathcal{D}\sqrt{d})^d}{\delta^d}\biggl(1-\frac{\kappa u^d}{\abs{\X}\lceil t\rceil}
            \biggr)^{\lceil t\rceil}\leq \frac{\pi_d(2\mathcal{D}\sqrt{d})^d}{\delta^d}\exp\biggl(- 
          \frac{\kappa u^d}{\abs{\X}}\biggr).
      \end{align*}
      Applying the union bound in \eqref{eq:Eunion} followed by these two inequalities, and then bounding the resulting sum 
      by a geometric series as in Proposition~\ref{prop:Rs.exponential.bound},
      shows that in either case we have
      $\P\bigl(\Phi(x,r,t,\mu)\bigr)\leq C e^{-c r^d}$ for constants $C$ and $c$.
      Now choose $r_0$ such that this upper bound is less than $\epsilon/2$ for $r>r_0$.
      
      Bounding $\Psi(x,r,t,\Ppp{t})$ and $\Psi(x,r,t,\bpp{t})$ is similar.
      If $\Psi(x,r,t,\mu)$ holds, then there must exist $\{w,z\}\subseteq B(x,r)$
      with $x\in\overline{S(w,z)}$ such that
          \begin{align*}
          S(w,z)\cap\mu\cap B(x,r)=\emptyset \qmq{but} S(w,z)\cap\mu \neq\emptyset.
          \end{align*}
      These relations imply that $S(w,z)$ extends outside of $B(x,r)$, which means that $S(w,z)$ has diameter at least $r$. Hence, by the scaled ball condition with the role of $\X$ played by 
      $t^{1/d}\X\cap B(x,r)=B(x,r)$,
      the set $S(w,z) \cap B(x,r)$ contains a ball of radius $\delta r/{\cal D}$.
      Thus, there exists a ball of radius $\delta r/{\cal D}$ containing no points of $\mu$, and one may now argue as for $\Phi(x,r,t,\mu)$.
    \end{proof}

\begin{proof}[Proof of Proposition~\ref{prop:Epm}]
For $\epsilon>0$ let $r_0$ be given as in Lemma \ref{lem:bd.OmPs}. 
    For $\mu=\Ppp{t}$ or $\mu=\bpp{t}$, for all $r\geq r_0$, $x\in \mathbb{R}^{d}$, and $t>t_0(x,r)$, it holds except on an event of probability at most $\epsilon$
    that $\mathcal{R}_S(x;\mu;t^{1/d}\X)=\mathcal{R}_S(x;\mu\cap B(x,r); B(x,r))$.
    By Lemma~\ref{lem:same.change.graph}, on this event
    $E^{\pm}(\mu)=E^{\pm}\bigl(\mu\cap B(x,r)\bigr)$.
\end{proof}

Since $\G(\Ppp{\infty})$ is an infinite graph, $L(\Ppp{\infty})$ does not exist in general.
However, when $E^{\pm}_x(\Ppp{\infty})$ is finite we may define $D_xL(\Ppp{\infty})$ by the difference
\begin{align*} 
  D_xL(\Ppp{\infty}) = \sum_{\{x,y\} \in E^+_x(\Ppp{\infty})}\psi(x,y) - \sum_{\{w,z\} \in E^-_x(\Ppp{\infty})}\psi(w,z).
\end{align*}
The following corollary implies that $D_xL(\Ppp{\infty})$ 
is also the almost surely finite limit of $D_xL(\Ppp{\infty}\cap B(x,r))$ as $r\to\infty$.
  \begin{corollary}\label{cor:weakstab} For all $x \in \RR^d$ the set $E_x^\pm(\Ppp{\infty})$ is finite almost surely, and for
    any $\epsilon>0$ there exists $r_0$ such that for all $r>r_0$ 
    \begin{align} \label{DxisDxB}
\P\Bigl( D_x L(\Ppp{\infty}) = D_x L(\Ppp{\infty}\cap B(x,r))\Bigr) &\geq 1-\epsilon.
    \end{align}
  \end{corollary}
  \begin{proof}
    Inequality~\eqref{eq:EpmPoisson} of Proposition~\ref{prop:Epm} with $t=\infty$ yields an $r_0$ such that $E_x^\pm(\Ppp{\infty})=E_x^\pm(\Ppp{\infty} \cap B(x,r))$ for all $x \in \mathbb{R}^d$ and $r>r_0$ with probability at least $1-\epsilon$, proving that \eqref{DxisDxB} holds. On the event
    that $D_x L(\Ppp{\infty}) = D_x L(\Ppp{\infty}\cap B(x,r))$, the quantity 
     $E_x^\pm(\Ppp{\infty})$ is finite. 
     Thus $E_x^\pm(\Ppp{\infty})$ is finite with probability at least 
     $1-\epsilon$. Since $\epsilon$ is arbitrary, $E_x^\pm(\Ppp{\infty})$ is finite with probability one.
  \end{proof}

We will use the next lemma to replace binomial processes with Poisson processes on large regions.
  \begin{lemma}\label{lem:Poisson.binomial.convergence}
    For any bounded measurable set $A \subseteq \mathbb{R}^d$, as $t\to\infty$
    \begin{align*}
      \bpp{t}\cap A \to \Ppp{\infty}\cap A
    \end{align*}
    in total variation.
  \end{lemma}
  \begin{proof}
    Let $M$ and $N$ be the number of points of $\bpp{t}$ and $\Ppp{t}$ that fall in $A$, respectively.
    Once $t$ is large enough that $A\subseteq t^{1/d}\X$, 
    the distribution of $M$ is $\Bin(t,\vol{A}/t)$, and the distribution of $N$ is 
    $\Poi(\vol{A})$. 
    It is well known that this binomial distribution converges in total variation to this Poisson
    distribution, and so $M$ and $N$ can be coupled so that they are equal with probability approaching
    $1$ as $t\to\infty$. As $\bpp{t}\cap A$ can be represented as $M$ points uniformly distributed over $A$
    and $\Ppp{\infty}\cap A$ as $N$ points uniformly distributed over $A$, the two point processes can
    be coupled to be equal with probability tending to $1$.
  \end{proof}

The next piece of the proof is to show that $D_xL(\Ppp{\infty})$ is nondeterministic. For any concrete
collection of forbidden regions, this is typically straightforward, but to show it in more
generality we need to present some technical arguments.

\begin{lemma}\label{lem:regular.points}
  Suppose $E=\interior \overline{E}$. Then for all $x\in\partial E$, every open neighborhood
  of $x$ intersects the interiors of $E$ and $E^c$.
\end{lemma}
\begin{proof}
  Let $x\in\partial E$ and let $U$ be an open neighborhood of $x$. By the definition of the boundary,
  $U$ intersects $E$ and $E^c$. Since $E$ is open, $E=\interior{E}$. Thus it just remains to show that $U$
  intersects $\interior{(E^c)}$. 
  
  Since $\overline{E}^c$ is an open set contained in $E^c$, we have $\overline{E}^c\subseteq\interior{(E^c)}$.
  Thus $\interior{(E^c)}^c\subseteq\overline{E}$.
  Now, suppose that $U$ does not intersect $\interior{(E^c)}$. 
  Then $U\subseteq\interior{(E^c)}^c\subseteq \overline{E}$. Since $U$ is open, we have 
  $U\subseteq\interior{(\overline{E})}=E$. Hence $x\in E$. But this contradicts $x\in\partial E$, since
  $E$ is open and hence contains none of its boundary.
\end{proof}

For a set $E\subseteq\RR^d$ and a direction $u\in \S^{d-1}:=\{u \in \mathbb{R}^d: \|u\|=1\}$, let $E_u=\{t\in[0,\infty)\colon
  tu\in E\}$, which one should think of as the one-dimensional slice of $E$ in direction~$u$. Let
  $\sigma$ denote uniform measure on $\S^{d-1}$.
\begin{lemma}\label{lem:fubini}
  Suppose that $E\subseteq\RR^d$ has Lebesgue measure zero. Then for $\sigma$-a.e.\ $u\in\S^{d-1}$,
  the set $E_u$ has one-dimensional Lebesgue measure zero.
\end{lemma}

\begin{proof}
  By \cite[Theorem~2.49]{Folland},
  \begin{align*}
    0=\int_{\RR^d}\1\{x\in E\}\,dx = C\int_{\S^{d-1}}\int_0^\infty \1\{r\in E_u\}r^{d-1}\,dr\,d\sigma(u),
  \end{align*}
  where $C$ is the volume of $\S^{d-1}$.
  This shows that the inner integrand is zero for $\sigma$-a.e.~$u$. As the inner integrand is zero if
  and only if $E_u$ has measure zero, this completes the proof.
\end{proof}

In the remainder of this section for the convenience we take $S(x,x)=\emptyset$ for all $x \in \mathbb{R}^d$. For instance, this convention allows us to write $x \in \mathbb{R}^d$ in place of $x \in \mathbb{R}^d\setminus \{y\}$ in the following lemma.
\begin{lemma}\label{lem:iso.measure.zero} Suppose that the forbidden regions $S(x,y)$ form 
  an $(S,u_0)$ regular isotropic family (see Definition~\ref{df:isotropic}).
  Then for any $w,y\in\RR^d$ with $w\neq y$, the set $\{x\in\RR^d \colon w\in\partial S(y, x)\}$
  has Lebesgue measure zero.
\end{lemma}
\begin{proof}\newcommand{\SO}{\mathrm{SO}}%
First note that by translation invariance of the forbidden regions,
  \begin{align*}
    \{x\in\RR^d\colon w\in\partial S(y, x)\} &= \{x\in\RR^d\colon w-y\in\partial S(0, x-y)\}\\
      &= \{x\in\RR^d\colon w-y\in\partial S(0, x)\}+y.
  \end{align*}
  Hence it suffices to prove that $\{x\in\RR^d\colon w\in\partial S(0, x)\}$ has measure zero
  for all $w\in\RR^d\setminus\{0\}$.
  
  The rest of the argument is easier to follow in $\RR^2$, and we present it there first. Let us identify
  $\RR^2$ with $\mathbb{C}$ for convenience. Observe that our isotropic assumption
  implies that $S(0,re^{i\theta})=re^{i\theta}S(0,1)$. Thus, with $T=S(0,1)$,  for any $w\in\mathbb{C}\setminus\{0\}$,
  \begin{align*}
    \int_{\RR^2} \1\{w\in\partial S(0,x)\}\,dx &=
      \int_0^{2\pi}\int_0^{\infty} \1\{r^{-1}e^{-i\theta}\in w^{-1}\partial T)\}\,r\,dr\,d\theta\\
      &=\int_0^{2\pi}\int_0^{\infty} \1\{te^{-i\theta}\in w^{-1}\partial T)\}\,t^{-3}\,dt\,d\theta,
  \end{align*}
  making the substitution $t=r^{-1}$. For a given $\theta$, the inner integrand is zero except when
  $t\in(w^{-1}\partial T)_{e^{-i\theta}}$, in the notation of Lemma~\ref{lem:fubini}.
  By our assumption in Definition~\ref{df:isotropic} that $S$ has negligible boundary,
  $w^{-1}\partial T$ has measure zero. Thus the inner integral is zero for a.e.~$\theta$ by
  Lemma~\ref{lem:fubini}, making the entire integral equal to zero.
  
  In higher dimensions, the proof is more complicated because rotation is more complicated, but the idea
  is the same.
  First, we record some facts about rotations of $\RR^d$ around the origin, 
  which can be represented as elements of $\SO(d)$, the special
  orthogonal group of order~$d$. The group $\SO(d)$ is isomorphic to
  $\S^{d-1}\times \SO(d-1)$. The decomposition works by specifying a vector $u\in\S^{d-1}$
  that a chosen vector $u_0$ is mapped to (note that we take this chosen vector to be the same as the
  axis of symmetry for the isotropic family), and then specifying how the orthogonal complement of the span of $u$ is
  rotated. As a corollary to this decomposition, if $u$ is chosen uniformly over $\S^{d-1}$,
  and the rotation of the orthogonal complement of $u$ is chosen from Haar measure on $\SO(d-1)$,
  then the result is distributed as Haar measure on $\SO(d)$. We let $\rho_u\in\SO(d)$ denote the rotation of 
  $\RR^d$ around the origin taking $u_0$ to $u$ by rotating the plane containing $u_0$ and $u$ 
  and fixing its orthogonal
  complement (if $u=u_0$, take $\rho_u$ to be the identity). We use the notation $\SO(u^\perp)$ to denote
  the subgroup of $\SO(d)$ fixing $u$, which as discussed above is isomorphic to $\SO(d-1)$.

  Let $\overline{x}\in \S^{d-1}$
  denote $x/\norm{x}$ for $x\neq 0$. Let $T=S+u_0/2=S(0,u_0)$.
  It follows from our isotropic assumption that
  \begin{align*}
    \partial S(0,x) = \norm{x} \rho_{\overline{x}}(\partial T).
  \end{align*}
  Thus, with $\sigma_d$ denoting Haar measure on $\S^d$, the measure of $\{x\in\RR^d\colon w\in\partial S(0, x)\}$ can be expressed as
  \begin{align*}
    \int_{\RR^d} \1\{w\in\norm{x} \rho_{\overline{x}}(\partial T)\}\,dx &=
      C\int_0^{\infty} \int_{\S^d} \1\{w\in r \rho_u(\partial T)\}r^{d-1}\,\,d\sigma_d(u)\,dr
  \end{align*}
  with the (irrelevant) constant determined by the volume of $\S^{d-1}$.
   Letting $\mu_u$ denote Haar measure on $\SO(u^\perp)$ normalized
  to have measure one, we can rewrite
  the integral as
  \begin{align*}
    C\int_0^\infty \int_{\S^{d-1}}&
      \int_{\SO(\mu^\perp)}\1\{w\in r \tau\rho_u(\partial T)\}r^{d-1}\,d\mu_u(\tau)\,d\sigma_{d-1}(u)\,dr\\
      &=C\int_0^\infty \int_{\S^{d-1}}
      \int_{\SO(u^\perp)}\1\{r^{-1}(\tau\rho_u)^{-1}(w)\in \partial T\}r^{d-1}\,d\mu_u(\tau)\,d\sigma_{d-1}(u)\,dr.
  \end{align*}
  As we mentioned before, $\tau\rho_u$ with $\tau$ distributed as $\mu_u$ and $u$ distributed
  as $\sigma_{d-1}$ is Haar-distributed over $\SO(d)$. By the invariance of Haar measure under 
  multiplication, 
  the distribution of $(\tau\rho_u)^{-1}(w)$ under this measure
  is uniform over $\norm{w}\S^{d-1}$. Hence we can rewrite the integral as
  \begin{multline*}
    \int_0^\infty \int_{\S^{d-1}}
      \1\{r^{-1}\norm{w}u\in \partial T\}r^{d-1}\,d\sigma_{d-1}(u)\,dr
        \\= \int_{\S^{d-1}}\int_0^\infty\1\{tu\in \norm{w}^{-1}\partial T\}t^{-(d+1)}\,dt\,
           d\sigma_{d-1}(u),
  \end{multline*}
  substituting $t=1/r$.
  The inner integral is supported on the ray $(\norm{w}^{-1}\partial T)_u$.
  Since the set $\norm{w}^{-1}\partial T$ has measure zero, the inner integral is thus zero
  for $\sigma$-a.e.~$u$ by Lemma~\ref{lem:fubini}.
\end{proof}

\newenvironment{step}[1]{\ignorespaces\par\medskip\par\goodbreak\noindent \textbf{Step #1}.\begingroup\itshape}{\endgroup\par\noindent\ignorespacesafterend}

\begin{lemma}\label{lem:nondet}
Assume the forbidden regions $S(x,y)$ are a $(S,u_0)$ regular isotropic family satisfying Assumption~\ref{assumption:specialy}, and that $S(x,y)=\interior\overline{S(x,y)}$ for all $\{x,y\}\subseteq\RR^d$.
Let $\{w,z\}\subseteq B(0,1)$ be distinct points.
Let $\mu$ be a homogeneous Poisson process on $\RR^d\setminus B(0,1+2\mathcal{D})$,
  and let $\mu'=\{w,z\}\cup\mu$.
  Then a.s.-$\mu$, 
  there exist open balls $A,A'\subseteq\RR^d$ such that
  \begin{align}
    D_xL(\mu') \neq D_{x'}L(\mu') \label{eq:nondet}
  \end{align}
  for all $x\in A$, $x'\in A'${, and furthermore the center and radii of $A$ and $A'$ are measurable random variables}.
\end{lemma}
\begin{proof} 
  Let $y\in\partial S(w,z)$ be a point satisfying $z\notin\partial S(w,y)$
  and $w\notin\partial S(z,y)$, whose existence is promised by Assumption~\ref{assumption:specialy}.
  The main idea of the proof is that adding to $\mu'$ any point close to $y$ has the same effect
  on $\G(\mu')$ except for possibly causing the deletion of the edge $wz$. Note that $wz$ is always
  present in $\G(\mu')$, as $S(w,z)$ has at most diameter $2\mathcal{D}$ and hence is
  contained in $B(0,1+2\mathcal{D})$, while $\mu'$ has no points in $B(0,1+2\mathcal{D})$ besides
  $w$ and $z$.
  
    \begin{step}{1} A.s.-$\mu$, we have $b\notin\partial S(y,a)$ 
    for all $\{a,b\}\subseteq \mu'$ with $a\neq b$.
    \end{step}    
    By Assumption \ref{assumption:specialy}, $w\notin \partial S(y,z)$ and $z\notin\partial S(y,w)$.
    Since $\partial S(y,z)$ and $\partial S(y,w)$ have measure zero, almost surely
    no points of $\mu$ fall in either of these sets. 
    Now we are left to show that 
    \begin{align} \label{eq:muonly}
      b\notin \partial S(y,a)\text{ a.s.,\quad for $a\in\mu,\,b\in\mu',\,a\neq b$.}
    \end{align}
    For a point process configuration $\chi$, let
    \begin{align*}
      f(\chi, a)=
      \# \Bigl( \bigl((\{w,z\}\cup\chi)\setminus \{a\}\bigr) \cap \partial S(y,a)\Bigr).
    \end{align*}
    Our goal is then to show that $\sum_{a\in\mu}f(\mu,a)=0$ a.s.
    By Mecke's formula,
    \begin{align*}
      \E \sum_{a\in\mu} f(\mu,a) = \E \int_{\RR^d \setminus B(0,1+2\mathcal{D})}f(\mu\cup\{a\}, a)\,da =
      \int_{\RR^d \setminus B(0,1+2\mathcal{D})} \E\Bigl[\#\bigl((\{w,z\}\cup\mu)\cap \partial S(y,a)\bigr)\Bigr]da,
    \end{align*}
    with the transposition of the integral and expectation justified by non-negativity of the integrand.
    For any $a\in\RR^d$, the set $\partial S(y,a)$ has measure zero by our
    assumption that $S$ has negligible boundary, and hence
    no points of $\mu$ are in $\partial S(y,a)$ a.s. Thus we can simplify the above expression to
    \begin{align} \label{eq:more.general.than.iso}
      \E \sum_{a\in\mu} f(\mu,a)     
      &=\int_{\RR^d} \#\bigl(\{w,z\}\cap \partial S(y,a)\bigr)\,da\\ \nonumber
      &=\int_{\RR^d} \bigl( \1\{w\in \partial S(y,a)\} + \1\{z\in\partial S(y,a)\}\bigr)\,da,
    \end{align}
    with the expectation removed because there is no longer any randomness in the integrand.
    Thus it follows from Lemma~\ref{lem:iso.measure.zero} that the
    integrand is zero except on a set of measure zero, proving that
    $\E \sum_{a\in\mu} f(\mu,a)=0$. This proves \eqref{eq:muonly}, completing
    the proof of this step.
    
    \begin{step}{2}
      A.s.-$\mu$, we have $y\notin \partial S(a,b)$ for $\{a,b\}\subset\mu'$, $\{a,b\}\neq\{w,z\}$.
    \end{step}
    This step follows by essentially the same proof as for Step 1.

    In the next step, we say that $E^+_x(\mu')$ and $E^+_y(\mu')$ are equivalent if the set of edges $E^+_x(\mu')$
    is equal to the set $E^+_y(\mu')$  when all edges of the form $\{y,a\}$ in the latter are replaced by $\{x,a\}$.
    Note that we do not
    need a definition like this for $E^-_x(\mu')$ and $E^-_y(\mu')$, since edges with vertices $x$ or $y$ do not appear
    in these collections.

     To prepare for the next step, recall that the Hausdorff metric between two subsets $A$ and $B$ of $\mathbb{R}^d$ is defined as
     	\beas
     	d_H(A,B)=\inf \{\epsilon>0: A \subseteq B_\epsilon, B \subseteq A_\epsilon\} \qmq{where}  F_\epsilon = \bigcup_{x \in F} \{y \in \mathbb{R}^d: \|y-x\| \le \epsilon\}.
     	\enas
     It is clear that when the forbidden regions form a regular isotropic family,
      the map $(x,y)\mapsto S(x,y)$ is Hausdorff continuous in $(x,y)  \in \mathbb{R}^d \times \mathbb{R}^d$.   
    \begin{step}{3}
  For some random radius $\rho>0$, it holds for all $x\in B(y,\rho)$ 
  that $E^+_x(\mu')$ is equivalent to $E^+_y(\mu')$, and that
  $E^-_x(\mu')$ is equal to either $E^-_y(\mu')$ or $E^-_y(\mu')\cup\{\{w,z\}\}$.
    \end{step}
    Let $R=\mathcal{R}_S(B(y,1); \mu'; \RR^d)$. The set $\mathcal{R}_S(B(y,1); \Ppp{\infty}; \RR^d)$
    is bounded a.s.-$\Ppp{\infty}$ by Proposition~\ref{prop:Rs.exponential.bound}. Since $\mu$ is distributed
    as $\Ppp{\infty}$ conditional on an event of positive probability,
    $\mathcal{R}_S(B(y,1); \mu; \RR^d)$ is bounded a.s.-$\mu$. As $\mu\subseteq \mu'$,
     Lemma~\ref{lem:monotonicity} shows that the set $R$ is bounded a.s.-$\mu$.      Recall by using \eqref{def:GSedge.set} that the addition of any point $x\in B(y,1)$ changes the graph $\G(\mu')$ only by the addition
    of edges $xa$ and deletion of edges $ab$ for $a,b\in R$.     

    Step~1 shows that for each $a\in\mu'\cap R$, the set $\partial S(y,a)$ does not contain any points
    of $(\mu'\setminus \{a\})\cap R$.   
    Since $(\mu'\setminus \{a\})\cap R$ is almost surely finite,
    the set $\partial S(y,a)$ has positive
    distance from $(\mu'\setminus\{a\})\cap R$, as both sets are compact.
  By the Hausdorff continuity of the map $S$, there is a positive distance $\rho^+_a$ such that
    for all $x\in B(y,\rho_a^+)$, the set $\partial S(x,a)$ avoids $(\mu'\setminus\{a\})\cap R$.
    Set $\rho^+$ to be the minimum of $\rho_a^+$ over the almost surely finitely many $a\in\mu'\cap R$.
    Then for all $x\in B(y,\rho^+)$, 
    the collections $E^+_x(\mu')$ and $E^+_y(\mu')$ are equivalent.  Standard continuity considerations yield that the $\rho _{a}$, and therefore $\rho $, can be built to be measurable random variables.
    
    Step~2 implies that for all $\{a,b\}\subseteq\mu'\cap R$ except for $\{w,z\}$, the set
    $\partial S(a,b)$ has a positive distance $\rho^-_{ab}$ from $y$. Set $\rho^-$ as the minimum of
    $\rho^-_{ab}$ over this almost surely finite collection of $\{a,b\}$. Then for $x\in B^o(y,\rho^-)$, as $y \in \partial S(w,z)$, and $S(w,z)$ is open,
    it holds that $E^-_x(\mu')$ is equal to either $E^-_y(\mu')$ or $E^-_y(\mu')\cup\{\{w,z\}\}$.
    Taking $\rho$ less than $\rho^+$ and $\rho^-$ completes the step.
    
    \begin{step}{4}
      Construction of $A,A'$ satisfying \eqref{eq:nondet}.
    \end{step}
    Let $A_{0}= B^o(y,\rho')\cap \interior{S(w,z)}$ and $A'_{0}=B^o(y,\rho')\cap\interior{(S(w,z)^c)}$ for $\rho' \in (0,\rho)$
    to be specified later.  By Lemma~\ref{lem:regular.points}, 
    both sets $A_{0}$ and $A'_{0}$ are open and nonempty, thus we define $A$ and $A'$ to be the balls with maximal radii centered respectively at arbitrary points $y_0 \in A_{0}$ and $y_0' \in A_{0}'$,
    chosen in some measurable way.
    By the previous step, $E^+_x(\mu')$ and $E^+_y(\mu')$
    are equivalent for $x\in A\cup A'$.
    For $x'\in A'$, we have $E^-_{x'}(\mu')=E^-_y(\mu')$, and for $x\in A$, we have 
    $E^-_x(\mu')=E^-_y(\mu')\cup \{\{w,z\}\}$. Thus for $x\in A$ and $x'\in A'$,
    \begin{align*}
      D_{x'}L(\mu') - D_{x}L(\mu') = \psi(w,z) +\sum_{a\colon \{a,x\}\in E^+_x(\mu')} \bigl(\psi(a,x')-\psi(a,x)\bigr).
    \end{align*}
    By  the continuity of $\psi$, and that $E_x^+(\mu')$ is finite, the sum can be made arbitrarily small over all $x\in A$, $x'\in A'$ by choosing $\rho'$ small enough, 
    a choice which can be made in a measurable way with respect to $\mu $.
    If we choose $\rho'$ to make the sum smaller than $\psi(w,z)$, non-zero by hypothesis as $w \neq z$, then \eqref{eq:nondet} holds
    for $x\in A$, $x'\in A'$.
\end{proof}

\begin{theorem} \label{thm:nondgeneracy}
 Assume that the forbidden regions $S(x,y)$ are a $(S,u_0)$ regular isotropic family satisfying Assumption~\ref{assumption:specialy}.
  Then for all $x\in \RR^d$, the random variable $D_x L(\Ppp{\infty})$ is nondeterministic.
\end{theorem}
\begin{proof}
  As $\interior{S(x,y)}\subseteq \interior{\overline{S(x,y)}}\subseteq \overline{S(x,y)}$,
  the sets $S(x,y)$ and $\interior{\overline{S(x,y)}}$ differ only on $\partial S(x,y)$, a set of
  measure zero. For each of the almost surely countably many pairs $\{a,b\}\subset\Ppp{\infty}$, there are almost surely no points of
  $\Ppp{\infty}$ on $\partial S(a,b)$ besides $a$ and $b$. Thus $\G(\Ppp{\infty})$ is almost surely
  unaffected by replacing each forbidden region $S(x,y)$ by $\interior \overline{S(x,y)}$.
  If $B=\interior\overline{A}$, then $B \subseteq \overline{A}$ hence $\overline{B} \subseteq \overline{A}$, and taking interiors and using that $B$ is open yields $B \subseteq \interior{\overline{B}} \subseteq \interior{\overline{A}}=B$, and thus $B=\interior\overline{B}$. Hence we can
  assume that $S(x,y)=\interior \overline{S(x,y)}$ for all $x,y$.
  
  Let $w$ and $z$ be chosen uniformly and independently from $B(0,1)$, and let
  $\mu$ be a homogeneous Poisson process with intensity~1 on $\RR^d\setminus B(0,1+2\mathcal{D})$. 
  With positive probability, $\Ppp{\infty}$ has exactly two points in $B(0,1+2\mathcal{D})$, both of which
  are contained in $B(0,1)$. Conditional on this event, $\Ppp{\infty}$ is distributed
  as $\mu':=\{w,z\}\cup\mu$. By Lemma~\ref{lem:nondet}, a.s.-$\mu$ 
  there exist open sets $A,A'\subseteq\RR^d$
  such that $D_x L(\mu')\neq D_{x'}L(\mu')$ for all $x\in A$ and $x'\in A'$. Thus, with positive probability,
  there exist open sets $A,A'\subseteq\RR^d$
  such that $D_x L(\Ppp{\infty})\neq D_{x'}L(\Ppp{\infty})$ for all $x\in A$ and $x'\in A'$.
  
  Suppose that $D_x L(\Ppp{\infty})=c$ a.s.\ for some $x\in\RR^d$ and some constant~$c$. 
  By the translation invariance of $\Ppp{\infty}$, this holds for all $x\in \RR^d$. Hence it holds
  almost surely that $D_x L(\Ppp{\infty})=c$ for all $x$ in a countable dense set of $\RR^d$.
  But this contradicts the conclusion of the previous paragraph.
\end{proof}

  We now use Theorem~\ref{thm:nondgeneracy} to show that if $x$ and $y$ are far enough apart, then
  with positive probability adding $x$ or $y$ to the process produces different effects on $L$.
  \begin{lemma}\label{lem:different.points.can.differ}
    Assume the conditions of Theorem~\ref{prop:variance}.
    There exist constants $a>b, r_0 \in (0,\infty)$ and $p_0 \in (0,1]$ such that for all $r>r_0$
    the following
    statement holds: for all $x,y\in\RR^d$, if the $r$-balls around $x$ and $y$ are disjoint and $t>t_1(x,y,r)=\max\{t_0(x,r),t_0(y,r),t_2(r)\}$ where $t_2$ is a function depending only on $r$, then
    \begin{align*}
      \P\bigl(\text{$D_x L(\mu) > a$ and $D_yL(\mu)<b$} \bigr) \geq p_0
    \end{align*}
    for $\mu=\Ppp{t}$ or $\mu=\bpp{t}$.
  \end{lemma}
  \begin{proof}
    Let first $\mu=\bpp{t}$.
    By Theorem~\ref{thm:nondgeneracy}, and that the distribution of $D_zL(\Ppp{\infty})$ does not depend on $z$ by translation invariance, there exist $a>b$ and $p>0$ such that
    for all $z\in\RR^d$,
        \begin{align*} 
        \P\bigl( D_zL(\Ppp{\infty}) > a \bigr)\geq p \qmq{and} 
        \P\bigl( D_zL(\Ppp{\infty}) < b \bigr) \geq p.
        \end{align*}
    Let $p_0 = (p-\epsilon)^2-3\epsilon$, choosing $\epsilon>0$ small enough that $p_0>0$.
    By Corollary~\ref{cor:weakstab}, for all sufficiently large $r$ and for all $z \in \mathbb{R}^d$ the random variables
    $D_z L(\Ppp{\infty})$ and $D_z L(\Ppp{\infty}\cap B(z,r))$ are within $\epsilon$ in total variation
    distance, and hence
    \begin{align}
      \P\bigl( D_zL(\Ppp{\infty}\cap B(z,r)) > a \bigr) \geq p-\epsilon\qmq{and}
        \P\bigl( D_zL(\Ppp{\infty}\cap B(z,r)) < b \bigr)\geq p-\epsilon.
        \label{eq:weakstab}
        \end{align}
       Next, from the total variation convergence given by invoking Lemma~\ref{lem:Poisson.binomial.convergence} with $A=B(x,r)\cup B(y,r)$, for all $r$ large enough that 
       \eqref{eq:weakstab} holds,
       and $t >t_2(r)$ depending on $r$, for any $\{x,y\} \subseteq \mathbb{R}^d$ satisfying $\|x-y\|>2r$,
         \begin{align}
         &\P\bigl( \text{$D_x L(\bpp{t}\cap B(x,r))>a$ and $D_yL(\bpp{t}\cap B(y,r))<b$}\bigr)\nonumber \\
         &\qquad\qquad\geq \P\bigl( \text{$D_x L(\Ppp{\infty}\cap B(x,r))>a$ and $D_yL(\Ppp{\infty}\cap B(y,r))<b$}\bigr)
         -\epsilon \nonumber \\
         &\qquad\qquad\geq (p-\epsilon)^2-\epsilon, \label{eq:UtandB}
         \end{align}
         with the last line following from \eqref{eq:weakstab} and the independence of
         $\Ppp{\infty}\cap B(x,r)$ and $\Ppp{\infty}\cap B(y,r)$.
        By Proposition~\ref{prop:Epm}, for all sufficiently large $r$ and all $t>\max\{t_0(x,r),t_0(y,r)\}$, it holds that
        \begin{align*}
        \P\bigl(D_x L(\bpp{t}\cap B(x,r)) = D_xL(\bpp{t}) \bigr) \ge 1-\epsilon \,\,\, \mbox{and} \,\,\,
        \P\bigl(D_y L(\bpp{t}\cap B(y,r)) = D_yL(\bpp{t}) \bigr) \ge 1-\epsilon.
        \end{align*}
        Hence, by a union bound, 
        \begin{align}
        \P\Bigl( \text{$D_x L(\bpp{t}\cap B(x,r)) = D_xL(\bpp{t})$ and 
        	$D_y L(\bpp{t}\cap B(y,r)) = D_yL(\bpp{t})$} \Bigr)\geq 1-2\epsilon.\label{eq:simadd}
        \end{align}
Now, taking any $r_0$ so that \eqref{eq:weakstab} and \eqref{eq:simadd} hold for all $r>r_0$, for all $t>t_1(x,y,r)$, by \eqref{eq:UtandB} and \eqref{eq:simadd},
    \begin{align*}
      \P\bigl( \text{$D_x L(\bpp{t})>a$ and $D_yL(\bpp{t})<b$}\bigr) \geq (p-\epsilon)^2-\epsilon-2\epsilon
       =p_0.
    \end{align*}

    The proof for the Poisson case is the same, except that the step involving 
    Lemma~\ref{lem:Poisson.binomial.convergence} is unnecessary.
  \end{proof}

We will need the following elementary lemma, which is essentially just Markov's inequality applied
  to a bounded random variable.
  \begin{lemma}\label{lem:markov}
    Suppose that $X$ is a random variable supported on $[0,n]$, and $\E X\geq np$. Then
    \begin{align}\label{eq:markov}
      \P\biggl(X > \frac{np}{2}\biggr) \geq \frac{p}{2-p}.
    \end{align}
  \end{lemma}
  \begin{proof}
    Let $Y=n-X$. Then $\E Y\leq n(1-p)$, and applying Markov's inequality to $Y$ yields
    \begin{align*}
      \P\biggl(X \leq \frac{np}{2}\biggr) &= \P\biggl(Y\geq n\Bigl(1-\frac{p}{2}\Bigr)\biggr)
        \leq \frac{1-p}{1-p/2},    
    \end{align*}
    yielding \eqref{eq:markov}.
  \end{proof}

In the remainder of this section let $a$, $b$, $r_0$, and $p_0$ be the constants given by Lemma \ref{lem:different.points.can.differ}.
For some $m>0$ and $1<r<\infty$, 
  we say that a pair of points $x$ and $y$ with $\norm{x-y}>2r$
  is \emph{$(m,r,t)$-influential} for $\mu$ if 
  \begin{enumerate}\renewcommand{\theenumi}{\alph{enumi}}\renewcommand*\labelenumi{(\theenumi)}
    \item[] \influential$_1(\mu)$: There exist sets $A\subseteq B(x,1)$ and $B\subseteq B(y,1)$
      each of Lebesgue measure $m$ such that
       $D_z L(\mu)>a$ for  $z\in A$ and $D_z L(\mu)<b$ for $z\in B$, and \label{infl1}
    \item[] \influential$_2(\mu)$: $R_S(B(x,1);\mu;t^{1/d}\X)\leq r$ and $R_S(B(y,1);\mu;t^{1/d}\X)\leq r$. \label{infl2}
  \end{enumerate}
  Note that a pair of influential points for $\mu$ are not required to be, and in fact will in general not be, points {\em of} $\mu$. We have made the radii of the balls containing $x$ and $y$ equal to $1$ in these definitions, but the value is unimportant.
  \begin{lemma} \label{lem:influential12}
    Assume the conditions of Theorem~\ref{prop:variance}.
    There exist constants $m \in (0,\infty),p \in (0,1]$ and $r \in (1,\infty)$ such that 
    if $x$ and $y$ are any two points 
    such that the $(r+1)$--balls centered around each are disjoint, then for all sufficiently large $t$
    \begin{align*}
      \P\bigl(\text{$(x,y)$ is $(m,r,t)$-influential for $\mu$}\bigr)\geq p
    \end{align*}
    for $\mu=\Ppp{t}$ and $\mu=\bpp{t}$.
  \end{lemma}
  \begin{proof}
By Proposition~\ref{prop:Rs.exponential.bound}, for all $\{x,y\} \subseteq \mathbb{R}^d$ and $t>\max\{t_0(x,r),t_0(y,r)\}$, as $r\to\infty$ the probability of $\influential_2(\mu)$ is lower bounded by a quantity tending to one, not depending on $\{x,y\}$.
    With $r_0$ and $p_0$ the constants given by 
        Lemma~\ref{lem:different.points.can.differ},
        let $p'_0=p_0/(2-p_0)$, and choose $r>r_0$ large enough that $\influential_2(\mu)$ holds
        with probability at least $1-p'_0/2$.
    Let $X$ and $Y$ be independent and
    distributed uniformly over $B(x,1)$ and $B(y,1)$, respectively. Let
    \begin{align}
      P(\mu) &:= \P\bigl(\text{$D_X L(\mu) > a$ and $D_YL(\mu)<b$} \mid \mu\bigr)\nonumber\\
        &\phantom{:}= \P\bigl(D_X L(\mu) > a\mid\mu\bigr)\P\bigl(D_YL(\mu)<b \mid \mu\bigr). 
           \label{eq:factors}
    \end{align}
    Note that
    \begin{align*}
      \P\bigl(D_X L(\mu) > a\mid\mu\bigr) &=
        \frac{\abs{\{z\in B(x,1)\colon D_z L(\mu)>a\}}}{\abs{B(x,1)}},
    \end{align*}
    with an analogous statement holding for the second factor in \eqref{eq:factors}.
      By Lemma \ref{lem:different.points.can.differ}, using that the $r$-balls around points in $B(x,1)$ and $B(y,1)$ do not intersect, by averaging $X$ and $Y$ over their supports we see that for $t>\sup_{u \in B(x,1), v \in B(y,1)}t_1(u,v,r)$ we have $\E P(\mu) \geq p_0$.
    Since $P(\mu)$ is supported on $[0,1]$, we apply Lemma~\ref{lem:markov} with $n=1$ and $p=p_0$ to conclude that
    $\P(P(\mu)>p_0/2) \geq p_0/(2-p_0)=p_0'$.
    If $P(\mu)\geq p_0/2$, then
    both factors in \eqref{eq:factors} are larger than $p_0/2$. Therefore, with probability
    at least $p_0'$, the pair $(x,y)$ satisfies $\influential_1(\mu)$
    with $m=p_0\abs{B(x,1)}/2$.
    
    Since $\influential_1(\mu)$ holds with probability at least $p'_0$ and $\influential_2(\mu)$
    holds with probability at least $1-p'_0/2$, by a union bound both hold simultaneously with 
    probability at least $p'_0/2$.
  \end{proof}
  From now on, we take $m$, $r$, and $p$ to be constants provided by Lemma \ref{lem:influential12}.

    \begin{lemma}\label{lem:influential}
    Assume the conditions of Theorem~\ref{prop:variance}.
    Let $\influential(\mu,t,\beta)$ be the event that there are at least $\beta t$ pairs of $(m,r,t)$-influential
    points for $\mu$, all of whose $(r+1)$-neighborhoods are disjoint and contained in $t^{1/d}\X$.
    For some $\beta,q>0$ independent of $t$, 
    for either $\mu=\Ppp{t}$ or $\mu=\bpp{t}$, it holds for all sufficiently large
    $t$ that
    \begin{align*}
      \P(\influential(\mu,t,\beta)) \geq q.
    \end{align*}
  \end{lemma}
  \begin{proof}
    For some $\beta'>0$, for all sufficiently large $t$ one can place at least $2\lceil\beta' t\rceil$ points in $t^{1/d}\X$
    so that all points have disjoint $(r+1)$-neighborhoods contained in $t^{1/d}\X$. 
    Let $n=\lceil\beta' t\rceil$, and arbitrarily form these $2n$ points into $n$ disjoint pairs.
    For large enough~$t$, by Lemma~\ref{lem:influential12},
    each pair has probability at least $p$ of being $(m,r,t)$-influential,
    so the expected number of such $(m,r)$-influential pairs is at least $np$.
    By Lemma~\ref{lem:markov}, there are at least $np/2$ such pairs with probability at least
    $p/(2-p)$. Now we can take $q=p/(2-p)$ and $\beta=p\beta'/3$, say.
  \end{proof}

\begin{proof}[Proof of Theorem~\ref{prop:variance}]
    It suffices to show that there exists $v$ such that $\var L(\mu)\geq vt$
    where  $\mu$ is either Poisson on $t^{1/d}\X$ with intensity~$1$ or binomial
    with $\lceil t\rceil$ points. Indeed, as $\psi(ax,ay)=a^\alpha \psi(x,y)$ for any $a>0$, we have $L(a\mu)=a^\alpha L(\mu)$, where $a\mu=\{ax, x \in \mu\}$. Hence, when $\var L(\mu)\geq vt$, scaling a process $\mu$ on $t^{1/d}\X$ to a process on $\X$, we have
    \begin{align*}
    	\var(L(t^{-1/d}\mu))= 	\var(t^{-\alpha/d}L(\mu)) = t^{-2\alpha/d} \var(L(\mu)) \ge v t^{1-2\alpha/d}.
    	\end{align*}

    The argument will go by splitting $\mu$ into a sum of independent point processes $\mu_1$ and $\mu_2$.
    Initially, take $\mu_1$ to be a deterministic set of points such that
    $\influential(\mu_1,t,\beta)$ holds for some $\beta>0$, and define $\mu_2$ as a point process
    on $t^{1/d}\X$ that is either Poisson
    with intensity $1/2$ or binomial with $\lfloor t/2\rfloor$ points. We start by arguing that
    $\var L(\mu_1\cup\mu_2)> Ct$ for some $C{  >0}$.
    
    Since $\influential(\mu_1,t,\beta)$ holds,    
    there exist point pairs $(x_1,y_1),\ldots,(x_n,y_n)$ with  $n \ge \beta t$
    with sets $A_i\subseteq B(x_i,1)$
    and $B_i\subseteq B(y_i,1)$ of measure~$m$ such that $\influential_1(\mu_1)$ and
    $\influential_2(\mu_1)$ hold for each pair.
    For some $\gamma>0$ to be specified, consider the event
   \begin{align*}
    F= \left\{ \bigl\lvert \left\{ 1 \le i \le n \colon  \bigl\lvert \mu_2\cap (B(x_i,r+1)\cup B(y_i,r+1)) \bigr\rvert =
    \bigl\lvert \mu_2\cap (A_i \cup B_i) \bigr\rvert=1 \right\} \bigr\rvert \ge \gamma n \right\}, 
    \end{align*}   
   that is, that for at least
    $\gamma n$ of the pairs $(x_i,y_i)$, 
   exactly one point of $\mu_2$ lands in the $(r+1)$-neighborhoods of $x_i$ and $y_i$, and it
    lands in either $A_i$ or $B_i$. We claim that $F$ occurs with positive probability not depending on $t$.
    Indeed, for any fixed $i$, the process $\mu_2$ will satisfy 
    \begin{align}
      \bigl\lvert \mu_2\cap (B(x_i,r+1)\cup B(y_i,r+1)) \bigr\rvert =
    \bigl\lvert \mu_2\cap (A_i \cup B_i) \bigr\rvert=1 \label{eq:one.in.each}
    \end{align}
    with at least some fixed, positive probability
    for large enough $t$. Choosing $\gamma$ small enough, the event $F$ then holds with some positive
    probability independent of $t$ by Lemma~\ref{lem:markov}.

    Now, the idea is that given that $\mu_2$ has exactly one point in either $A_i$ or $B_i$, 
    it is equally likely to be in either. Conditional on $F$, we then essentially have 
    $\gamma n=\Omega(t)$ coin flips, each contributing a constant term to $\var L(\mu_1\cup\mu_2)$.
    To put this into practice, we
    partition $\mu_2$ into $\{X_1,\ldots,X_l\}$ and $\{Y_1,\ldots,Y_{l'}\}$, where
    the first set consists of the points of $\mu_2$ that are contained in $A_i\cup B_i$
    for some $i$ satisfying \eqref{eq:one.in.each}. Thus $l\geq\gamma n$ when $F$ holds.
    Now, let $\widetilde{\mu}=\mu_1\cup\{Y_1,\ldots,Y_{l'}\}$, and
    express $L(\mu_1\cup \mu_2)$ as the telescoping sum
    \begin{align*}
      L(\mu_1\cup \mu_2) &= L(\widetilde{\mu}) + D_{X_1}L(\widetilde{\mu}) + D_{X_2}L(\widetilde{\mu}\cup\{X_1\})
               +\cdots + D_{X_l}L(\widetilde{\mu}\cup\{X_1,\ldots,X_{l-1}\}).
    \end{align*}
    By $\influential_2(\mu_1)$, for any $1\leq j\leq l$ we have $R_S(X_{j};\mu_1;t^{1/d}\X)\leq r$.
    Because $X_j$ satisfies \eqref{eq:one.in.each} for some $i$,
    all points of $\mu_2$ except for $X_{j}$ lie outside of $B(X_{j},r)$.
    By \eqref{eq:add.points} of Lemma \ref{lem:stabilization},
    \begin{align*}
      D_{X_{j}}L(\widetilde{\mu}\cup\{X_1,\ldots,X_{j-1}\}) = D_{X_{j}}L(\mu_1).
    \end{align*}
    Thus we can rewrite
    $L(\mu_1\cup\mu_2)$ as
    \begin{align}\label{eq:sum.of.independents}
      L(\mu_1\cup\mu_2) &= L(\widetilde{\mu}) + D_{X_1}L(\mu_1) + D_{X_2}L(\mu_1) + \cdots + D_{X_l}L(\mu_1).
    \end{align}
    By construction, $X_{j}$ falls into $A_i\cup B_i$
    for exactly one choice of $i$.
    Conditional on $F$, the point $X_{j}$ is equally likely to be in $A_{i}$ or $B_{i}$.
    Furthermore, which of these it lands in is independent for $1\leq j\leq l$ conditional on $F$.
    If $X_{j}$ lands in $A_{i}$, then $D_{X_{j}}L(\mu_1)>a$, and if $X_{j}$ lands in $B_{i}$, then
    $D_{X_j}L(\mu_1)<b$, by the definition of $\influential_1(\mu_1)$.
    Thus, \eqref{eq:sum.of.independents} expresses $L(\mu_1\cup\mu_2)$
    as a sum of terms that are conditionally independent given $F$ and $\widetilde{\mu}$
    and which each have conditional variance bounded from below, showing that
    \begin{align*}
      \var \Bigl( L(\mu_1\cup\mu_2) \ \Big\vert\  \1_F,\widetilde{\mu} \Bigr)
        \geq Cl \geq C\gamma n \geq C\gamma\beta t
    \end{align*}
    on the event $F$, for some absolute constant $C>0$.
    As $F$ occurs with probability that can be bounded away from zero uniformly for all $t$, and
    \begin{multline*}
      \var L(\mu_1\cup\mu_2)  = \E \var \Bigl( L(\mu_1\cup\mu_2) \ \Big\vert\  \1_F,\widetilde{\mu} \Bigr)+\var \E  \Bigl( L(\mu_1\cup\mu_2) \ \Big\vert\  \1_F,\widetilde{\mu} \Bigr) \\ \geq \E \var \Bigl( L(\mu_1\cup\mu_2) \ \Big\vert\  \1_F,\widetilde{\mu} \Bigr),
    \end{multline*}
    we have shown that $\var L(\mu_1\cup\mu_2)$ grows at least
    as a constant times $t$.

    To complete the proof, we now let $\mu_1$ be a point process on $t^{1/d}\X$, independent of $\mu_2$,
    and either Poisson with intensity~$1/2$ or binomial with $\lceil t/2 \rceil$ points.
    Thus $\mu$ can be expressed as $\mu_1\cup \mu_2$.
    By Lemma~\ref{lem:influential}, for all $t$ sufficiently large, the event $\influential(\mu_1,t,\beta)$ holds with probability at least $q$ for some $\beta,q>0$ not depending on $t$. 
    (Strictly speaking, we replace $\X$ by $2^{1/d}\X$ and $t$ by $t/2$ when applying 
    Lemma~\ref{lem:influential}.)
    By the previous argument, the variance of $L(\mu)$ conditional on $\influential(\mu_1,t,\beta)$
    for sufficiently large $t$
    is at least $Ct$ for a constant $C>0$ not depending on $t$, from which the theorem follows. 
  \end{proof}

    \bibliographystyle{plain}
    \bibliography{gbg}

\begin{thebibliography}{10}

\bibitem{AlSh10}
David~J. Aldous and Julian Shun.
\newblock Connected spatial networks over random points and a route-length
  statistic.
\newblock {\em Statist.\ Sci.}, 25(3):275--288, 2010.

\bibitem{Cha08}
Sourav Chatterjee.
\newblock A new method of normal approximation.
\newblock {\em Ann.\ Probab.}, 36(4):1584--1610, 2008.

\bibitem{Dev88}
Luc Devroye.
\newblock The expected size of some graphs in computational geometry.
\newblock {\em Comput.\ Math.\ Appl.}, 15(1):53--64, 1988.

\bibitem{englund1981remainder}
Gunnar Englund.
\newblock A remainder term estimate for the normal approximation in classical
  occupancy.
\newblock {\em Ann.\ Probab.}, 9(4):684--692, 1981.

\bibitem{Folland}
Gerald~B. Folland.
\newblock {\em Real analysis}.
\newblock Pure and Applied Mathematics (New York). John Wiley \& Sons, Inc.,
  New York, second edition, 1999.
\newblock Modern techniques and their applications, A Wiley-Interscience
  Publication.

\bibitem{GoldsteinPenrose}
Larry Goldstein and Mathew~D. Penrose.
\newblock Normal approximation for coverage models over binomial point
  processes.
\newblock {\em Ann. Appl. Probab.}, 20(2):696--721, 2010.

\bibitem{LRP}
Rapha\"el Lachi\`eze-Rey and Giovanni Peccati.
\newblock New {K}olmogorov bounds for functionals of binomial point processes.
\newblock to appear in Ann. Appl. Probab., available at arXiv:1505.04640, 2015.

\bibitem{last2014normal}
G{\"u}nter Last, Giovanni Peccati, and Matthias Schulte.
\newblock Normal approximation on {P}oisson spaces: Mehler's formula, second
  order {P}oincar{\'e} inequalities and stabilization.
\newblock {\em Probability Theory and Related Fields}, pages 1--57, 2014.

\bibitem{LP}
G{\"u}nter Last and Mathew~D. Penrose.
\newblock Poisson process {F}ock space representation, chaos expansion and
  covariance inequalities.
\newblock {\em Probab.\ Theory Related Fields}, 150(3-4):663--690, 2011.

\bibitem{PenYuk01}
Mathew~D. Penrose and J.~E. Yukich.
\newblock Central limit theorems for some graphs in computational geometry.
\newblock {\em Ann.\ Appl.\ Probab.}, 11(4):1005--1041, 2001.

\bibitem{Schreiber}
Tomasz Schreiber.
\newblock Limit theorems in stochastic geometry.
\newblock In {\em New perspectives in stochastic geometry}, pages 111--144.
  Oxford Univ. Press, Oxford, 2010.

\end{thebibliography}

\end{document}